\documentclass[11pt, reqno]{amsart}
\usepackage{xcolor}
\usepackage{soul}

\usepackage{amsmath}
\usepackage{a4wide}
\usepackage{latexsym}
\usepackage{cite}
\usepackage{graphicx}
\usepackage{float}
\usepackage{amsxtra}
\usepackage{xcolor}
\usepackage{amsthm}
\usepackage{epsfig}
\usepackage{amsfonts}
\usepackage{amssymb}
\usepackage{amsmath}
\usepackage{cases}
\usepackage{longtable}
\usepackage{graphicx}
\usepackage{tikz-cd}
\usepackage[pagebackref, hypertexnames=false, colorlinks, citecolor=red, linkcolor=blue, urlcolor=red]{hyperref}

\usepackage{amscd}
\usepackage{ifthen}
\usepackage{enumerate}

\theoremstyle{plain}

\newtheorem{thm}{Theorem}[section]
\newtheorem{cor}[thm]{Corollary}
\newtheorem{lem}[thm]{Lemma}
\newtheorem{prop}[thm]{Proposition}
\newtheorem{defn}[thm]{Definition}

\theoremstyle{definition}

\newtheorem{rem}[thm]{Remark}
\newtheorem{ex}[thm]{Example}

\numberwithin{equation}{section}

\newtheorem*{thmA}{Theorem}

\newcommand{\overbar}[1]{\mkern1.5mu\overline{\mkern-1.5mu#1\mkern-1.5mu}\mkern 1.5mu}

\newenvironment{psmallmatrix}
 {\left(\begin{smallmatrix}}
 {\end{smallmatrix}\right)}

\newcommand{{\ran}}{\mbox{\rm ran}~}

\begin{document}
\title[Geometric invariants for submodules of analytic Hilbert modules]{Geometric invariants for a class of submodules of analytic Hilbert modules via the sheaf model}

\author[S. Biswas]{Shibananda Biswas}
\address[S. Biswas]{Department of Mathematics and Statistics, Indian
Institute of Science Education and Research Kolkata, Mohanpur 741246, Nadia, West
Bengal, India}
\email[S. Biswas]{shibananda@iiserkol.ac.in}
\author[G. Misra]{Gadadhar Misra}
\address[G. Misra]{Statistics and Mathematics Unit, Indian Statistical Institute, Bangalore 560059, and Department of Mathematics, Indian Institute of Technology, Gandhinagar 382055}
\email[G. Misra]{gm@isibang.ac.in}

\author[S. Sen]{Samrat Sen}
\address[S. Sen]{4L B.G.Bye Lane, Naktala, Kolkata 700047}
\email[S. Sen]{samratsen@alum.iisc.ac.in}

\begin{abstract}
Let $\Omega \subseteq \mathbb C^m$ be a bounded connected open set and $\mathcal H \subseteq \mathcal O(\Omega)$ be an analytic Hilbert module, i.e., the Hilbert space $\mathcal H$ possesses a reproducing kernel $K$, the polynomial ring $\mathbb C[\boldsymbol{z}]\subseteq \mathcal H$ is dense and the point-wise multiplication induced by $p\in  \mathbb C[\boldsymbol{z}]$ is bounded on $\mathcal H$. We fix an ideal $\mathcal I \subseteq \mathbb C[\boldsymbol{z}]$ generated by $p_1,\ldots,p_t$ and let $[\mathcal I]$ denote the completion of $\mathcal I$ in $\mathcal H$. The sheaf $\mathcal S^\mathcal H$ associated to analytic Hilbert module $\mathcal H$ is the sheaf $\mathcal O(\Omega)$ of holomorphic functions on $\Omega$ and hence is free. However, the subsheaf $\mathcal S^{\mathcal [\mathcal I]}$ associated to $[\mathcal I]$ is coherent and not necessarily locally free. 
Building on the earlier work of \cite{BMP},  we prescribe a hermitian structure for a coherent sheaf and use it to find tractable invariants. Moreover,  we prove that if the zero set $V_{[\mathcal I]}$ is a submanifold of codimension $t$, then there is a unique local decomposition for the kernel $K_{[\mathcal I]}$ along the zero set that serves as a holomorphic frame for a vector bundle on $V_{[\mathcal I]}$. The complex geometric invariants of this vector bundle are also unitary invariants for the submodule $[\mathcal I] \subseteq \mathcal H$.

\end{abstract}

\thanks{The second-named author would like to acknowledge funding through the J C Bose National Fellowship and the MATRICS grant for his research.}
\thanks{A number of the results presented in this paper are from the PhD thesis of the third named author submitted to the Indian Institute of Science, Bangalore}
\subjclass[2020]{Primary: 47B13, 47B32, 47B35, 32A10, 32A36, 32A38}

\keywords{Hilbert module, reproducing kernel function, Analytic Hilbert module, submodule, resolution, holomorphic Hermitian vector bundle, coherent sheaf, linear space}
\dedicatory{Dedicated to the memory of J\"{o}rg Eschmeier}

\maketitle
\section{Introduction}
The notion of a Hilbert module over a function algebra was introduced by R. G. Douglas in the late eighties. Over the past couple of decades,   problems of multi-variate operator theory have been discussed using the language of Hilbert modules. In this paper,  we continue this tradition.  Let us begin by setting up some conventions that will be in force throughout.
\begin{enumerate}
\item $\mathbb C[\boldsymbol{z}]:=\mathbb C[z_1, \ldots , z_m]$ is the polynomial ring in $m$ variables.
\item $\Omega \subseteq \mathbb C^m$ is a bounded domain.
\item $\mathcal O(\Omega)$ is the ring of holomorphic functions on $\Omega$.
\item $\mathcal H$ is a complex separable Hilbert space and $\mathcal L(\mathcal H)$ is the algebra of bounded linear operators on $\mathcal H$.
\end{enumerate}
First, we recall some useful notions and terminology that we would be using through out this paper. 
For $1\leq i \leq m$, let $T_i:\mathcal H \to \mathcal H$ be a  commuting set of bounded linear operators on a Hilbert space $\mathcal H$ and let $\boldsymbol{T}$ denote the commuting tuple $(T_1, \ldots , T_m)$.
For a polynomial $p\in \mathbb C[\boldsymbol{z}]$, the map $p \to m_p:=p(T)$ defines a homomorphism from $\mathbb C[\boldsymbol{z}]$ to $\mathcal L(\mathcal H)$. Thus the map $(p,h) \to p(T) h$, $h \in \mathcal H$, defines module multiplication on the  Hilbert space $\mathcal H$ over the polynomial ring  $\mathbb C[\boldsymbol{z}]$. 

The Hilbert space $\mathcal H$ equipped with this multiplication is said to be a {\sf \emph{Hilbert module}} over the polynomial ring $\mathbb C[\boldsymbol{z}]$.

A closed subspace $\mathcal H_0 \subseteq \mathcal H$ is said to be a {\sf \emph{submodule}} of $\mathcal H$ if it is invariant under the module multiplication, i.e., $m_p f \in \mathcal H_0$ for all $f\in \mathcal H_0$. The {\sf \emph{quotient module}} $\mathcal Q$ is the quotient space $\mathcal H / {\mathcal H_0}$, which is the ortho-complement of $\mathcal H_0$ in $\mathcal H$. The module multiplication on this space is defined by compression of the multiplication on $\mathcal H$ to $\mathcal Q$, i.e., $\mathfrak m_p f = P_{\mathcal H_0^\perp}\big ( m_p f \big )$, $f\in \mathcal Q$.

(The original definition of the Hilbert module required the module map to be continuous in both the variables, however, here we dispense with this requirement.)

Two Hilbert modules $\mathcal H$ and $\tilde{\mathcal H}$ are said to be ``unitarily'' equivalent if there exists a unitary module map $\theta: \mathcal H \to \tilde{\mathcal H}$ intertwining the module maps, $m$ and $\tilde{m}$, that is, $ \tilde{m}_p \theta = \theta m_p$.

\begin{defn}
A Hilbert module $\mathcal H\subseteq \mathcal O(\Omega)$ over the polynomial ring $\mathbb C[\boldsymbol{z}]$
is said to be an {\sf \emph{analytic Hilbert module}} if it possesses a reproducing kernel $K$ and the polynomial ring $\mathbb C[\underline{z}]$ is included in $\mathcal H$ and it is dense. In particular, $K(w,w) \not = 0$, $w\in \Omega$.
\end{defn}

Fix an ideal $\mathcal I \subseteq \mathbb C[\boldsymbol{z}]$ generated by $p_1,\ldots,p_t$ and let $[\mathcal I]$ denote the completion of $\mathcal I$ in $\mathcal H$. 
Thus
\[
\begin{tikzcd}
0 \arrow{r} &\mbox{[}.  
\mathcal I \mbox{]}
\arrow{r}
& {\mathcal H} \arrow{r}
& \mathcal Q \arrow{r}& 0 ,
\end{tikzcd} \]
is a short exact sequence of Hilbert modules,
where $\mathcal Q$ is the quotient module $[\mathcal I]^\perp$. 
The rigidity phenomenon for Hilbert modules is the question of determining when two submodules of the form $[\mathcal I]$,  $[\mathcal I']$ of an analytic Hilbert module are equivalent. There is also the question of finding invariants for the sub and quotient modules $[\mathcal I]$, $[\mathcal I]^\perp$. In this paper, we mostly discuss the pair $(\mathcal H, [\mathcal I])$, where $\mathcal H$ is an {\sf \emph analytic Hilbert modules}. 

For any submodule $\mathcal M$ of an analytic Hilbert module, we define the subsheaf
$\mathcal S^\mathcal M$ of the sheaf $\mathcal O(\Omega)$ of holomorphic functions on $\Omega$ as follows: 
$$\mathcal S^\mathcal M(U) := \Big \{\, \sum_{i=1}^n ({f_i|}_U) h_i : f_i \in \mathcal M, h_i \in \mathcal O(U), n\in\mathbb N\,\Big  \},$$
for any open $U \subset \Omega$. We note $\mathcal S^\mathcal H(U) = \mathcal O(U)$ and thus it is free and naturally gives rise to a holomorphic line bundle on $\Omega$. However, in general, the sheaf corresponding to the sub-module $[\mathcal I]$ is not free, not even locally free, but only coherent.

As we have seen, the analytic Hilbert modules give rise to a holomorphic hermitian line bundle. Consequently, the methods of Cowen and Douglas developed for the class $B_1(\Omega)$ in \cite{CD} applies. The class $B_1(\Omega)$ consists of those commuting tuples of operators $\boldsymbol T:=(T_1, \ldots , T_m)$ on a Hilbert space $\mathcal H$  possessing an open set of eigenvalues, say $\Omega\subset \mathbb C^m$, a holomorphic map $\gamma:\Omega \to \mathcal H$ with the property that $(T_i-w)\gamma(w) =0$, $1\leq i \leq m$, for each $w\in\Omega$ and finally $\gamma(w)$ modulo the constants is the only such vector. 
Clearly, the map $w\to \gamma(w)$ defines a trivial holomorphic line bundle $\mathcal L$ on $\Omega$ with a hermitian structure inherited from the Hilbert space $\mathcal H$, namely, $\|\gamma(w)\|$. 
Recall that the curvature of $\mathcal L$ is defined to be the $(1,1)$ form:
$$\mathcal K_{\mathcal{L}}(z) := \sum_{i,j=1}^n \frac{\partial^2}{\partial z_i \partial \bar{z}_j} \log \|\gamma(z)\|^2 dz_i \wedge d\bar{z}_j.$$
The Cowen-Douglas theorem says that the curvature is a complete invaraint for the commuting tuple of operators $\boldsymbol T$. Thus two analytic Hilbert modules are equivalent if and only if their curvatures are equal. However, there is a large class of Hilbert modules, such as $[\mathcal I]$, where  the dimension of the joint kernel $\cap_{i=1}^m \ker (M_i - w_i)^*$ is not constant. For instance,  
let $H^2_{(0,0)}(\mathbb D^2)$ be the submodule of functions vanishing at $(0,0)$ of the Hardy module $H^2(\mathbb D^2)$.
In this case, 
$\dim \cap_{i=1}^2 \ker (M_i - w_i)^* = 1$ if $(w_1,w_2) \not = (0,0),$ 
while it is $2$ if $(w_1,w_2) = (0,0)$, cf. \cite{DMV}.

We study a class of Hilbert modules very similar to the ones introduced in \cite{SB} and designated $\mathfrak B_1(\Omega)$ in that paper. Here, while we retain the same symbol, the definition instead of requiring $\dim \big ( \mathcal H {\big /} \big ( \sum_{i=1}^m (z_i - w_i) \mathcal H \big ) \big ) < \infty,$ merely requires that $\dim \bigcap_{i=1}^m \ker (M_i - w_i)^* < \infty$, $w\in \Omega.$ For what we do here, this is adequate, and more importantly, this condition is not very hard to verify as in Proposition \ref{[I]}. 

\begin{defn}
The class $\mathfrak B_1(\Omega)$ consists of Hilbert modules $\mathcal H \subseteq \mathcal O(\Omega)$ possessing a reproducing kernel $K$ and such that  $\dim \cap_{i=1}^m \ker (M_i - w_i)^* < \infty$, $w\in \Omega$.
\end{defn}

All the analytic Hilbert modules $\mathcal H \subseteq \mathcal O(\Omega)$ are in $\mathfrak B_1(\Omega)$, see Section 2. However, the reproducing kernel $K$ of a Hilbert module in $\mathfrak B_1(\Omega)$ may vanish, that is, $K(w,w) = 0$ for $w$ in some closed subset of $\Omega$,  unlike the case of the analytic Hilbert modules. Indeed the modules in this class are the ones where the dimension of the joint kernel $\cap_{i=1}^m \ker (M_i - w_i)^*$ of the module multiplication is not necessarily constant. Therefore the techniques from complex geometry developed in \cite{CD, CS} do not apply directly.
The typical example that we will be considering is a pair $(\mathcal H, [\mathcal I])$, where $\mathcal H$ is an analytic Hilbert module and  $[\mathcal I]$ is  the completion of a polynomial ideal $\mathcal I $ in the norm topology of $\mathcal H$. If $\mathcal I = <p>$ is a principal ideal, then the kernel $K_{[\mathcal I]}$ factors:  $K_{[\mathcal I]}(z,w) = p(z)\chi(z, w)\overline{p(w)}$.  Thus  $w\mapsto \chi(\cdot, w)$ is holomorphic frame defining a holomorphic hermitian line bundle which serves as an invariant for $[\mathcal I]$. However, to study modules of the form $[\mathcal I]$ without assuming that $\mathcal I$ is necessarily principal, we make crucial use of the decomposition theorem from \cite{BMP}, which is reproduced below. 
\begin{thm}[Theorem 1.5, \cite{BMP}]\label{sourcethm} Let $w_0$ be a fixed but arbitrary point in $\Omega$. Suppose $\mathcal M$ is in $\mathfrak B_1(\Omega)$ and let $K$ be the reproducing kernel of $\mathcal M$. 
Assume that $g_{i}^{0}$, $1\leq i\leq d$, is a minimal set of generators for the stalk $\mathcal{S}^{\mathcal{M}}_{w_{0}}$. Then
\begin{enumerate}
\item[\rm(i)] there exists an open neighbourhood $\Omega_{0}$ of $w_{0}$ such that
$$K(\cdot,w)=\overline{g_{1}^{0}(w)}K^{(1)}(w)+\cdots+\overline{g_{d}^{0}(w)}K^{(d)}(w),~w\in\Omega_{0}$$
for some choice of anti-holomorphic maps $K^{(1)},\ldots,K^{(d)}:\Omega_{0}\to\mathcal{M}$,
\item[\rm(ii)] the vectors $K^{(i)}(w)$, $1\leq i\leq d$, are linearly independent in $\mathcal{M}$ for $w$ in some small neighbourhood of $w_{0}$,
\item[\rm(iii)] the vectors $K^{(i)}(w_{0})$, $1\leq i\leq d$, are uniquely determined by the generators $g_{1}^{0},\ldots,g_{d}^{0}$,
\item[\rm(iv)] the linear span of the set of vectors $\{K^{(i)}(w_{0}):1\leq i\leq d\}$ in $\mathcal{M}$ is independent of the choice of generators $g_{1}^{0},\ldots,g_{d}^{0}$, and
\item[\rm(v)] $M_{p}^{*}K^{(i)}(w_{0})=\overline{p(w_{0})}K^{(i)}(w_{0})$, for $i=1,\ldots,d$, where $M_{p}$ denotes the module multiplication by the polynomial $p$.
\end{enumerate}
\end{thm}

The decomposition  of the kernel function given in Theorem \ref{sourcethm} naturally gives rise to the hermitian structure on the ``linear space'' associated to the coherent sheaf $\mathcal S^{[\mathcal I]}$ described in Section 3. Moreover, it also enables us to define a holomorphic hermitian line bundle in a canonical manner leading to tractable invariants for the submodule $\mathcal M$. 




On the other hand, part (iii) of the Theorem \ref{sourcethm} gurantees only that  $K^{(i)}$'s are uniquely determined at the point $w_0$. It is natural to ask what condition on an ideal $\mathcal {I}$ might ensure that these vectors are determined uniquely over the zero set $V_{[\mathcal I]}$. In Section 4, we prove that if $V_{[\mathcal I]}$ is  a complex manifold codimension $t$ and the ideal is generated by $t$ polynomials, then the $K^{(i)}$'s are uniquely determined. This is part of Theorem \ref{mainthm} which improves Theorem \ref{sourcethm} in this case. In Corollary \ref{cor1}, we have shown that the dimension formula holds for a larger class ideals than covered in Duan-Guo  \cite[Proposition 2.2 and Theorem 2.3]{DG}. In section 5, Theorem \ref{thm2}, we show that the class of the vector bundle on the zero set $V_{[\mathcal I]}$ obtained from Theorem \ref{mainthm} is an invariant for the Hilbert module $[\mathcal I]$. We further show in Corollary \ref{cor:4.6}, that in case of a homogeneous ideal, the curvatures are related via conjugation by a constant matrix with respect to two different minimal sets of generators.

\section{Preliminaries} 
The relationship between analytic Hilbert modules and commuting tuples of operators in the Cowen-Douglas class is revealed by using the notion of a non-negative definite kernel $K$ with the reproducing property discussed below. In particular, $w \to K(\cdot, w)$, $w\in \Omega$, serves as an anti-holomorphic frame for the eigenbundle for the operator (or, the commuting tuple of operators) in the Cowen-Douglas class.

\begin{defn}
The function $K:\Omega \times \Omega \to \mathbb C$ is said to be {\sf \emph{non-negative definite kernel}} if:
$$\Big \langle \big (\! \big ( K(z_i,z_j) \big ) \! \big ) {x} , {x} \Big \rangle \geq 0, \,\, \{z_1, \ldots, z_n\} \subseteq \Omega,\,\, {x}\in \mathbb C^n,\,\, n\in \mathbb N.$$
\end{defn}
In this paper,  we assume the kernel function $K$ to be holomorphic in first variable and anti-holomorphic in second.  Let $k_w$ be the holomorphic function defined by $k_w(z):=K(z, w)$.

Let $\mathcal H^0$ be the linear span of the vectors $\{k_w: w \in \Omega\}$. For any finite subset $\{z_1, \ldots , z_n\}$ of $\Omega$ and complex numbers $x_1, \ldots , x_n$,  set
$$ \Big \| \sum_{i=1}^n x_j k_{z_j} \Big \|^ 2 : = \Big \langle \big (\! \big (  K(z_i,z_j) \big ) \! \big ) {x} , {x} \Big \rangle,$$
where ${x}$ is the vector whose {\emph i}\,-th coordinate is $x_i$. Since $K$ is assumed to be non-negative definite, this defines a semi-norm on the linear space $\mathcal H^0$.

Now, the sesquilinear form $K(z,w) = \langle K(\cdot, w), K(\cdot, z) \rangle$ is non-negative definite by assumption. However, for $f\in \mathcal H^0$, the Cauchy-Schwarz inequality gives
$$|f(w)|^2 = |\langle f, K(\cdot, w)\rangle|^2 \leq \|f\|^2 K(w,w),\,\, w\in \Omega.$$
It follows that if $\|f\|=0$, then $f$ is the zero vector in $\mathcal H^0$.
Thus the semi-norm defined by $K$, as above, is indeed a norm on $\mathcal H^0$.  The completion of $\mathcal H^0$ equipped with this norm is a Hilbert space, which consists of holomorphic functions on $\Omega$ (cf. \cite{aronszajn, PR}).  The function $k_w:=K(\cdot, w)$, then has the reproducing property, namely, $$\langle f, k_w\rangle = f(w),\,\, f\in \mathcal H,\,\, w\in \Omega.$$

Conversely, assume that the point evaluation $e_w:\mathcal H \to \mathbb C$, $w\in \Omega$, on a Hilbert space $\mathcal H \subseteq \mathcal O(\Omega)$  is bounded, that is, $|f(w)| \leq C \|f\|$, $f\in \mathcal H$. Then $f(w) = \langle f , k_w \rangle$ for some vector $k_w\in \mathcal H$.  It follows that $e_w^* = k_w$. Let  $ K(z,w) = e_z k_w  = e_z e_w^*$. The function $K$ is holomorphic in the first variable and anti-holomoprhic in the second.  Also,  $\overline{K(z,w)} = K(w,z)$. Finally, for any finite subset $\{z_1, \ldots , z_n\}$ of $\Omega$, we have
\begin{equation}\label{rk}
0 \leq \Big \| \sum_{i=1}^n x_j k_{z_j} \Big \|^2
 = \sum_{i,j=1}^n \bar{x}_i x_j K(z_i,z_j)
 = \big \langle \big (\! \big (  K(z_i,z_j) \big ) \! \big ) {x} , {x} \big \rangle, \,\, x\in \mathbb C^n.
\end{equation}
The non-negative definite function $K$ is said to be the {\sf \emph{ reproducing kernel}} of the Hilbert space $\mathcal H$.


There are several notions, namely, locally free modules \cite{XChenRGD}, modules with sharp kernels \cite{AS}, quasi-free modules \cite{qfree}, which are closely related to the notion of analytic Hilbert modules. In all of these variants, the definition ensures the existence of a holomorphic hermitian vector bundle corresponding to these Hilbert modules.  The fundamental theorem of Cowen and Douglas (cf. \cite{CD}) then applies and says that the equivalence class of the Hilbert modules and those of the vector bundles determine each other.  Finding tractable invariants for these remains a challenge.

It is easy to verify that $M_p^* k_w = \overline{p(w)} k_w$, or equivalently, $k_w$ is in $\ker (M_p - p(w))^*$. If $\mathcal H$ is an analytic Hilbert module, then it follows that the $\dim \cap_{i=1}^m \ker (M_i - w_i)^* = 1,$ $w\in \Omega$, where $M_i$ is the operator of multiplication by the coordinate function $z_i$ on $\mathcal H$.
This is easily verified as follows. For any $f\in  \cap_{i=1}^m \ker (M_i - w_i)^*$, $p\in \mathbb C[\underline{z}]$, we have
$$
\langle f, p\rangle = \langle M_p^* f , 1 \rangle = \langle \overline{p(w)} f, 1 \rangle = \langle a k_w , p \rangle,
$$
where $a= \langle f, 1 \rangle$.   Therefore, if $\mathcal H$ is an analytic Hilbert module, then the dimension of the joint kernel $\cap_{i=1}^m \ker (M_i - w_i)^*$ is $1$ and is spanned by the vector $k_w$. Hence the map $\gamma: \Omega^* \to \mathcal H$, $\gamma(w) = k_{\bar{w}}$ is holomorphic for $w \in \Omega^*:= \{w \in \mathbb C^m: \bar{w} \in \Omega\}$. Thus it defines a holomorphic hermitian line bundle $\mathcal L$ on $\Omega^*$.
If $\alpha$ is a non-vanishing holomorphic function defined on $\Omega^*$, then $\alpha(w) \gamma(w)$ serves as a holomorphic frame for the line bundle $\mathcal L$ as well. The hermitian structures induced by these two holomorphic frames are
$\|\gamma(w)\|^2$ and $|\alpha(w)|^2 \|\gamma(w)\|^2$, respectively. These differ by the absolute square of a non-vanishing holomorphic function. However, the curvature $\mathcal K_{\mathcal{L}}$ defined relative to either one of these two frames is the same and therefore serves as an invariant for the holomorphic hermitian line bundle $\mathcal L$.

  It is easy to see that $[\mathcal I]$ is in $\mathfrak B_1(\Omega)$ using the notion of \textit{module tensor product}.  Let $\mathcal H_1$ and $\mathcal H_2$ be two Hilbert modules over the polynomial ring $\mathbb C[\underline{z}]$. The Hilbert space tensor product $\mathcal H_1 \otimes \mathcal H_2$ of these two Hilbert modules has two natural module multiplications, namely, $m_p\otimes {\rm Id} (f_1\otimes f_2) = m_p (f_1) \otimes f_2$ and  ${\rm Id} \otimes m_p (f_1\otimes f_2) = f_1 \otimes m_p(f_2)$. The module tensor product $\mathcal H_1 \otimes_{\mathbb C[\underline{z}]} \mathcal H_2$ is obtained by identifying the space on which these two multiplications coincide. Set $\mathcal N$ to be the closed subspace spanned by the set of vectors  
$$\{ m_p f_1 \otimes f_2 - f_1 \otimes m_p f_2 : f_1 \in \mathcal H_1, f_2\in \mathcal H_2, p \in \mathbb C[\underline{z}]\} \subseteq \mathcal H_1 \otimes \mathcal H_2.$$
The subspace $\mathcal N$ is a submodule for both the left and the right multiplications: $m_p\otimes{\rm Id}$ and ${\rm Id} \otimes m_p$. On the quotient $\mathcal N^\perp =(\mathcal H_1 \otimes \mathcal H_2) \ominus \mathcal N$, these two module multiplications coincide (cf. \cite{DP}). The quotient Hilbert space $\mathcal N^\perp$ equipped with this multiplication is the module tensor product.

We consider the special case  $\mathcal H\otimes_{\mathbb C[\underline{z}]} \mathbb C_w$, where $\mathbb C_w$ is $\mathbb C$ equipped with the module multiplication $m_p(\lambda) = p(w) \lambda$, $w \in \Omega$, $p\in \mathbb C[\underline{z}]$. 
In $\mathcal H$, let $\mathcal{J}(w)$ denote  the joint kernel $\cap_{i=1}^m \ker (M_i - w_i)^* = \cap_{p\in \mathbb C[\underline{z}]} \ker (M_p - p(w))^*$. We have the following  useful lemma.
\begin{lem}
Let $\mathcal H \subseteq \mathcal O(\Omega)$ be a Hilbert module.
For any $w\in \Omega$, we have the equality
$$\mathcal H\otimes_{\mathbb C[\underline{z}]} \mathbb C_w =  \Big (\mathcal J(w)\Big )\otimes \mathbb C.$$
\end{lem}
\begin{proof}
The proof consists of the following string of equalities:
\begin{align*}
\mathcal H \otimes_{\mathbb C[\underline{z}]} \mathbb C_w &= \big ( \mathcal H\otimes \mathbb C\big )/\big \{p f \otimes \lambda - f \otimes p(w) \lambda: f\in \mathcal H, p \in \mathbb C[\underline{z}], \lambda \in \mathbb C\big \} \\
&=\big \{ (p-p(w)) f \otimes \lambda:f\in \mathcal H, p \in \mathbb C[\underline{z}], \lambda \in \mathbb C \big \}^\perp\\
&= \big \{ g \otimes \mu \in \mathcal H\otimes \mathbb C: \langle g, (p-p(w))f \rangle \mu \bar{\lambda} = 0,\,\,f\in \mathcal H, p \in \mathbb C[\underline{z}], \lambda \in \mathbb C\big \}\\
&=\big \{g\otimes \mu:\langle M_{p-p(w)}^* g, f \rangle = 0, \,\,f\in \mathcal H, p \in \mathbb C[\underline{z}], \lambda \in \mathbb C \big \}\\
&=\Big (\mathcal J(w)\Big )\otimes \mathbb C.
\end{align*}
These equalities are  easily verified.
\end{proof}

\begin{prop} \label{[I]}
The submodule $[\mathcal I]$ of an analytic Hilbert module   $\mathcal H \subseteq \mathcal O(\Omega)$ is in $\mathfrak B_1(\Omega)$.
\end{prop}
\begin{proof} We note that 
\begin{align*}
\dim \bigcap_{i=1}^m \ker (M_i - w_i)^* &= \dim \mathcal J(w)  = \dim \big ([\mathcal I] \otimes_{\mathbb C[\underline{z}]} \mathbb C_w\big ).
\end{align*}
Since the polynomial ring is Noetherian, it follows that $[\mathcal I]$ is finitely generated.  Now, from \cite[Lemma 5.11]{DP}, it follows  that the $\dim \big ([\mathcal I] \otimes_{\mathbb C[\underline{z}]} \mathbb C_w\big )$ is finite,  completing the proof.
\end{proof}

Any module map $L:\mathcal H \to \widetilde{\mathcal H}$ must map the joint kernel $\mathcal J(w) \subseteq \mathcal H$ into the joint kernel $\widetilde{\mathcal J}(w) \subseteq \widetilde{\mathcal H}$. If the map $L$  is assumed to be invertible then its restriction to the kernel $\mathcal J(w) \subseteq \mathcal H$ is evidently an isomorphism. Thus we have proved the following proposition.
\begin{prop}
Suppose $\mathcal H$ and $\widetilde{\mathcal H}$ are two Hilbert module in $\mathcal{O}(\Omega)$, which are isomorphic via an invertible module map. Then $\mathcal H\otimes_{\mathbb C[\underline{z}]}\mathbb C_w$ and $\widetilde{\mathcal H}\otimes_{\mathbb C[\underline{z}]}\mathbb C_w$ are isomorphic for each $w\in \Omega$.
\end{prop}
Therefore, $\dim \mathcal J(w)$ is clearly an invariant for the class of Hilbert modules in $\mathcal{O}(\Omega)$. For an analytic Hilbert module, this is a constant function.

Now we observe (as in \cite{DMV}), for the Hardy module $H^2(\mathbb D^2)$, $\dim \mathcal J(w)$ is identically $1$ for all $w\in \mathbb D^2$ while for the submodule $H^2_{(0,0)}(\mathbb D^2)$, it equals $1$ for $w\not =(0,0)$ but is equal to $2$ at $(0,0)$. Thus $H^2(\mathbb D^2)$ and  $H^2_{(0,0)}(\mathbb D^2)$ are not equivalent via any invertible module map.  

The module tensor product $\mathcal H\otimes_{\mathbb C[\underline{z}]} \mathbb C_w$ is said to be the {\sf \emph{localization}} of $\mathcal H$ at $w$ and the set $ {\rm Sp}(\mathcal H):=\big \{\mathcal H\otimes_{\mathbb C[\underline{z}]} \mathbb C_w: w\in \Omega\big \}$ is said to be the spectral sheaf. When $\mathcal H\subseteq \mathcal O(\Omega)$ is an analytic Hilbert module,  the spectral sheaf determines an anti-holomorphic line bundle via the frame $1\otimes _{\mathbb C[\underline{z}]} 1_w$. The hermitian structure is induced from $\mathcal H\otimes_{\mathbb C[\underline{z}]}\mathbb C_w$. In general, however, the spectral sheaf is a direct sum of $k$ copies of $\mathbb C_w$, where $k$ is between $1$ and $t$, which is the rank of $\mathcal H$, see below. In what follows, it will be convenient to use the notion of {\sf \emph{locally free}} module of rank $n$ over $\Omega^*:=\{w\in\mathbb{C}^{m}:\bar{w}\in\Omega\}$, where $\Omega$ is some open bounded subset of $\mathbb C^m$.

\begin{defn}[Definition 1.4, \cite{CD}]\label{locally-free}
Let $\mathcal H$ be a Hilbert module over $\mathbb C[\underline{z}].$
Let $\Omega$ be a bounded open connected subset of $\mathbb C^m.$
We say $\mathcal H$ is locally free of rank $n$ at $w_0$ in $\Omega^*$ if  there exists a neighbourhood $\Omega_0^*$ of $w_0$ and holomorphic functions $\gamma_1, \gamma_2, \ldots , \gamma_n:\Omega_0^* \to \mathcal H$ such that the linear span of the set of $n$ vectors $\{\gamma_1(w), \ldots, \gamma_n(w)\}$  is the module tensor product $\mathcal H \otimes_{\mathbb C[\underline{z}]} \mathbb C_{\bar{w}}.$ Following the terminology of \cite{XChenRGD}, we say that a module $\mathcal H$ is {\textit locally free on $\Omega^*$} of rank $n$  if it is locally free of rank $n$ at every $w$ in $\Omega^*.$
\end{defn}

Setting $V_{[\mathcal I]}: = \big\{ z\in \Omega: f(z) =0, f\in [\mathcal I]\big \}$, in the examples $(\mathcal H,[\mathcal I])$, we have the following Lemma from \cite[Lemma 1.3]{BMP}.
\begin{lem}
The submodule $[\mathcal I]$ of an analytic Hilbert module  $\mathcal H\subseteq\mathcal{O}(\Omega)$ is locally free on $\Omega^*$ of rank $1$ if the  ideal $\mathcal I$ is principal while if $p_1, \ldots p_t$, $t > 1$, is a minimal set of generators for $\mathcal I$, then $[\mathcal I]$ is locally free on $(\Omega\setminus V_{[\mathcal I]})^*$ of rank $1$.
\end{lem}

Assume that there is finite set of generators for $[\mathcal I]$, moreover, let $\{p_1, \ldots, p_t\}$ be a minimal set of generators. A description of the spectral sheaf  for a pair $(\mathcal H, [\mathcal I])$ is given below. 
For $w \in \Omega \setminus V_{[\mathcal I]}$,  we have $[\mathcal I]\otimes_{\mathbb C[\underline{z}]}\mathbb C_w = p_i \otimes_{\mathbb C[\underline{z}]} 1_w$, $1\leq i \leq t$. However,  note that
\begin{align*}
p_i \otimes_{\mathbb C[\underline{z}]} 1_w &= P_{\mathcal J(w)\otimes \mathbb C} (p_i \otimes 1)\\
&=\big (P_{\mathbb C[{K_{[\mathcal I]}(\cdot, w)]}}\otimes 1 \big )(p_i \otimes 1)\\
&=\tfrac{p_i(w)}{K_{[\mathcal I]}(w, w)} K_{[\mathcal I]}(\cdot, w)\otimes 1.
\end{align*}
Here $\mathbb C[{K_{[\mathcal I]}(\cdot, w)]}$ denotes the one dimensional space spanned by the vector $K(\cdot,w)$.  Thus the set of vectors $p_i \otimes_{\mathbb C[\underline{z}]} 1_w$ are linearly dependent and therefore $\dim [\mathcal I]\otimes_{\mathbb C[\underline{z}]}\mathbb C_w  = 1$ for $w\in \Omega \setminus V_{[\mathcal I]}$.  Based on this observation and explicit computations in simple examples, it was conjectured in \cite{DMV} that
$$\dim [\mathcal I]\otimes_{\mathbb C[\underline{z}]}\mathbb C_w   = \begin{cases} 1 & \mbox{ for } w\in \Omega \setminus V_{[\mathcal I]}\\
 \mbox{\rm codim of }V_{[\mathcal I]}& \mbox{ for } w\in V_{[\mathcal I]}
\end{cases}$$
This formula is shown to be false in general by means of several examples in the paper \cite{DG} by Duan and Guo. They show that the formula given above is valid if the ideal $\mathcal I$ has any one of the following properties:
\begin{enumerate}
\item $\mathcal I$ is singly generated,
\item $\mathcal I$ is a prime ideal in $\mathbb C[z_1, z_2]$
\item $\mathcal I$  is a prime ideal in $\mathbb C[z_1,\ldots , z_m]$, $m > 2$ and $w$ is a smooth point of $V_{[\mathcal I]}$.
\end{enumerate}

\section{Hermitian structure}


Let $\mathcal H\subseteq \mathcal O(\Omega)$  be an {\tt analytic}  Hilbert module over the polynomial ring $\mathbb C[\boldsymbol z]$, i.e., it is a reproducing kernel Hilbert space such that the polynomial ring is densely contained in it.  Let $\mathcal M:=[\mathcal I]$ be the completion of an ideal $\mathcal I \subset \mathcal H$. 
Thus $\mathcal M$ is again a reproducing kernel Hilbert space. However, the reproducing kernel $K(\cdot, w)$ of the submodule $\mathcal M$ might be  the zero vector for some $w\in\Omega$. Consequently, the dimension of the eigenspace $\cap_{i=1}^m \ker (M_i^* - \bar{w}_i)$ is no longer necessarily independent of $w\in \Omega$. For an analytic Hilbert module $\mathcal H$, the eigenspaces over $w\in \Omega$ defines a holomorphic hermitian vector bundle. As we have noted, any submodule $\mathcal M$ of the form $[\mathcal I]$ of an analytic Hilbert module $\mathcal H$ is in the class $\mathfrak B_1(\Omega)$. While $\mathcal M$ need not define a holomorphic vector bundle, it was shown in \cite{BMP} that it defines a coherent sheaf $\mathcal S^\mathcal M(\Omega)$. We assume that $\mathcal M$ is a submodule of an analytic Hilbert module and with a slight abuse of notation, we let 
$\mathcal S$ denote the coherent sheaf $\mathcal S^\mathcal M(\Omega)$ defined by $\mathcal M$. 

In what follows, we make the standing assumption that $\mathcal M$ is a submodule of an analytic Hilbert module $\mathcal H$ of the form $[\mathcal I]$.  

The main goal here is to define a hermitian structure for the coherent sheaf $\mathcal S$ by taking advantage of its presentation via the short exact sequence: Since the sheaf $\mathcal S$ is a subsheaf of the sheaf of holomorphic functions $\mathcal O(\Omega)$, for a fixed $w_0\in\Omega$, there is a open neighbourhood $U_0$  of $w_0$ and a natural number $n$ such that $\mathcal O_{U_0}^n \to \mathcal S_{U_0}\to 0$ is exact. The decomposition theorem \ref{sourcethm} produces linearly independent set of vectors  $\{K^{(1)}(w), \ldots, K^{(n)}(w) \}$, $w\in U_0$, depending holomorphically on $w$. The size $n$ of this set of vectors depends on $w_0$. Thus we have a Hermitian structure on the free module $\mathcal O_{U_0}^n$. We describe below how to transplant this Hermitian structure to the stalk of $\mathcal S$ at $w_0$.

For each $w\in \Omega$, let $\mathfrak m_w$ be the maximal ideal of $\mathcal O_w$ and let $\mathcal Q_w:=\mathcal O_w/\mathfrak m_w$ be the quotient module. For any  analytic sheaf $\mathcal S$, let  $\overline{\mathcal S}_{w}$ be the module $\mathcal S_w\otimes_{\mathcal O_w}\mathcal Q_w = \mathcal S_w/\mathfrak m_w \mathcal S_w$. If $\mathcal S$ is coherent, then $\overline{\mathcal S}_w$ is a finite dimensional linear space.

If $\mathcal O_U^n \to \mathcal S_U\to 0$ is an exact sequence, then the induced map $\overline{\mathcal O}^n_w \to \overline{\mathcal S}_w \to 0$, $w\in U$, is also exact. If $\langle \cdot, \cdot \rangle_w$ is an inner product on $\overline{\mathcal O}^n_z\, (\cong \mathbb C^n)$, then we equip $ \overline{\mathcal S}_w$ with the inner product induced by the onto map  
$\overline{\mathcal O}_w \to \overline{\mathcal S}_w \to 0$. 
\begin{defn}
 A hermitian structure on a coherent sheaf $\mathcal S$ is an assignment of inner products on each $\overline{\mathcal S}_w$, $w\in \Omega$. 
 
 If $\mathcal S$ is a coherent sheaf over $\Omega$, then there exists an open neighbourhood $U$ of a fixed but arbitrary  $w_0\in \Omega$ and a map $\varphi$ such that $\mathcal O_U^n \stackrel{\varphi}{\to} \mathcal S_U\to 0$ is exact. 
 
 A hermitian structure on a coherent sheaf $\mathcal S$ is said to be smooth if for a fixed but arbitrary $w_0\in \Omega$, there is a neighbourhood $U$ of $w_0$ and a smooth inner product on $\mathcal O^n(U)$, that is, a smooth assignment of inner products on the locally free sheaf $\overline{\mathcal O}^n_w$, $w\in U$, such that the inner product on $\overline{\mathcal S}_w$ is the quotient inner product induced by the onto map $\overline{\mathcal O}_w \stackrel{\overline{\varphi}_w}{\longrightarrow} \overbar{\mathcal S}_w \to 0$.  
 
 Finally, a coherent sheaf equipped with a smooth hermitian structure, as above, is said to be a holomorphic hermitian sheaf.
\end{defn}

For any fixed but arbitrary $w_0\in \Omega$, let $\Omega_0 \subseteq \Omega$ be an  open neighbourhood of $w_0$. 
Pick $h_1, \ldots, h_d$ in $\mathcal O(\Omega_0)$ and set $[h_1 g^0_1 + \cdots + h_d g_d^0]$ to be the equivalence class of 
$h_1 g^0_1 + \cdots + h_d g_d^0$ in $\mathcal S_{w_0}/\mathfrak m_{w_0} \mathcal S_{w_0}$. The assignment  
$$\big [h_1 g^0_1 + \cdots h_d g_d^0 \big ] \mapsto \big ( h_1(w_0), \ldots , h_d(w_0) \big )$$
is well defined: 

Suppose that $\sum_{i=1}^m h_i g_i^0 = \sum_{i=1}^m h^\prime_i g_i^0$. Then $\sum_{i=1}^m (h_i-h_i^\prime) g_i^0$ belongs to $\mathfrak m_{w_0} \mathcal S(K)_{w_0}$ since $h_i(0) = h_i^\prime(0)$, $i=1, \ldots , m$. Therefore, $\big [\sum_{i=1}^m h_i g_i^0 \big ] = \big [\sum_{i=1}^m h^\prime_i g_i^0\big ]$. 

This assignment provides a linear isomorphism between $\overline{\mathcal S}_{w_0}$ and $\mathbb C^d$. 

For any fixed but arbitrary $w_0\in \Omega$, there is an open neighbourhood $\Omega_0 \subseteq \Omega$ of $w_0$ such that
$K(\cdot , w) = g^0_1(w) K_{w_0}^{(1)}(\cdot, w) + \cdots + g^0_d(w) K_{w_0}^{(d)}(\cdot, w)$, $w\in \Omega$, where 
$\{g^0_1, \ldots , g^0_d\}$ are a minimal set of generators for the stalk $\mathcal S_{w_0}$ at $w_0\in \Omega_0$. Also, for any $w\in \Omega_0$, the vectors $\{K^{(1)}_{w_0}(\cdot , w), \ldots , K^{(d)}_{w_0}(\cdot , w)\} \subset \mathcal M$ are  linearly independent.

Now, we can assign an inner product on $\overline{\mathcal S}_{w_0}$ as follows: Let  
$$H(w_0) =\big ( \!\! \big ( H_{ij}(w_0) \big ) \!\! \big ): = \big ( \!\! \big ( \big \langle K^{(i)}_{w_0}(\cdot, w_0) , K^{(j)}_{w_0}(\cdot, w_0) \big \rangle \big ) \!\! \big )$$
be the Grammian  of the linearly independent vectors $\{K^{(1)}_{w_0}(\cdot, w_0), \ldots , K^{(d)}_{w_0}(\cdot ,w_0)\}$. The matrix $H(w_0)$ is therefore positive definite and is invertible. We define a hermitian structure on $\overline{\mathcal S}_{w_0}$ using $H(w_0)$. Assertion (iv) of Theorem \ref{sourcethm} shows that the hermitian structure at $w_0$ is independent of the choice of the generators.  

We claim that the prescription of a hermitian structure on $\overline{\mathcal S}_w$, $w$ in some open neighbourhood $\Omega_0$ of $w_0\in\Omega$ using $H(w)$, is a smooth hermitian structure on $\mathcal S$. 
We have a resolution $\mathcal O^d\to \mathcal S \to 0$. Let us define  an inner product on  $\overline{\mathcal O}^d_w$ using the hermitian structure $\mathbb H(w)^{-1}$, where $\mathbb H(w) = \big ( \!\! \big ( \big \langle K^{(i)}_{w_0}(\cdot , w), K^{(j)}_{w_0}(\cdot ,w) \big \rangle \big ) \!\! \big )$ and the set $\{K^{(1)}_{w_0}(\cdot, w), \ldots , k^{(d)}_{w_0}(\cdot ,w)\}$, $w\in \Omega_0$, of linearly independent vectors are as in  Theorem \ref{sourcethm} except that we have added the subscript $w_0$ to emphasize the dependence of these vectors on the fixed but arbitrary point $w_0$.  

We prove that the inner product induced by $\mathbb H$  on 
$\overline{\mathcal S}_{w_0}$, $w_0\in \Omega_0$, coincides with the one we have already defined.

For any $w\in \Omega_0$, let $\{s^{(w)}_1, \ldots , s^{(w)}_e\}$ be a minimal set of generators of $\mathcal S_w$. We can then find linearly independent vectors $K^{(1)}_w(\cdot ,u), \ldots , K^{(e)}_w(\cdot ,u)$ that are anti-holomorphic in the variable $u$ defined on some open neighbourhood of $w$ such that $K(\cdot, u) = \sum_{j=1}^e s^{(w)}_j(u) K^{(j)}_w(\cdot , u)$. Since $\{g_1^{(0)}, \ldots , g_d^{(0)}\}$ is also a set of generators for $\mathcal S_w(K)$, not necessarily minimal, we must have a full rank matrix $G \in \mathcal M_{e \times d}(\mathcal O_w)$ such that 
$$(s^{(w)}_1(u), \ldots , s^{(w)}_e(u)) G(u) = (g_1^{(0)}(u), \ldots ,g_d^{(0)}(u)),$$ 
u in some small open neighbourhood of $w$. 

\begin{thm} The coherent sheaf $\mathcal S$ is a holomorphic hermitian sheaf, that is, 
$$\big (\overline{G(w)}\, \mathbb H(w) \, G(w)^\dagger \big ) = \big ( \!\! \big ( \big \langle K^{(i)}_w(\cdot , w) , K^{(j)}_w(\cdot, w)  \big \rangle \big )\!\!\big )_{i,j=1}^e,\,\, w\in \Omega_0,$$
where $G(w)^\dagger$ is the transpose of the matrix $G(w)$ and $e$ depends on $w$. 
\end{thm}
\begin{proof}
 For $\alpha \in \mathcal O_w^d$, set $\beta:= G \alpha$, and note that  
\begin{align*}
   \begin{psmallmatrix} g^{(0)}_1& \ldots & g^{(0)}_d \end{psmallmatrix} \begin{psmallmatrix} \alpha_1\\ \vdots \\ \alpha_d \end{psmallmatrix}  & = \begin{psmallmatrix} s^{(w)}_1(w)& \ldots & s^{(w)}_e(w)\\ \end{psmallmatrix} G(w) \begin{psmallmatrix} \alpha_1\\ \vdots\\ \alpha_d \end{psmallmatrix}\\
   &= \begin{psmallmatrix} s^{(w)}_1(w)& \ldots & s^{(w)}_e(w) \end{psmallmatrix} \begin{psmallmatrix} \beta_1\\ \vdots\\ \beta_d \end{psmallmatrix}.
\end{align*}
Suppose that $\overline{\mathcal S}_w$ is identified with $\mathbb C^e$ via the map: $\big [ \sum_{i=1}^e \beta_i s_i^{(w)} \big ] \mapsto \big (\beta_1(w), \ldots, \beta_e(w) \big )$. Then the hermitian structure on $\overline{\mathcal S}_w$, by definition, is induced by the 
non-negative definite matrix $\big ( \!\! \big ( \big \langle K^{(i)}_{w}(\cdot , w), K^{(j)}_{w}(\cdot ,w) \big \rangle \big ) \!\! \big )_{i,j=1}^e$. 
The map 
$\mathbb C^d \,\cong\, \overline{\mathcal O}^d_w \,\stackrel{\overline{\varphi}_w}{\longrightarrow}\, \overline{\mathcal S}_w \,\cong\, \mathbb C^e$ 
is then given by the formula:
$$\begin{tikzcd} (\alpha_1(w), \ldots , \alpha_d(w)) \arrow{r} & \big ( \sum_{i=1}^d \alpha_i g_i^{(0)} \big )_{|w} \arrow{r} & \big ( \sum_{i=1}^d \beta_i s_i^{(w)} \big )_{|w} 
\arrow{r}& (\beta_1(w), \ldots , \beta_e(w)). \end{tikzcd}$$
and $(\beta_1(w) , \ldots , \beta_e(w)) = (\alpha_1(w), \ldots , \alpha_d(w) ) G(w)^\dagger$. Setting the hermitian structure on $\mathbb C^d$, induced by the set of linearly independent vectors $\{K^{(1)}_{w_0}(\cdot, w), \ldots , k^{(d)}_{w_0}(\cdot ,w)\}$ to be $\mathbb H(w):=\big ( \!\! \big ( \big \langle K^{(i)}_{w_0}(\cdot , w), K^{(j)}_{w_0}(\cdot ,w) \big \rangle \big ) \!\! \big )$, we see that the hermitian structure on $\mathbb C^e$ obtained via  the map $\overline{\varphi}_w$ is of the form $\big (\overline{G(w)}\mathbb H(w) G^\dagger(w)\big )^{-1}$. Since 
$$
 s^{(w)}_1(w) K^{(1)}_w(\cdot , w) +\cdots + s^{(w)}_e(w) K^{(e)}_w(\cdot , w)= K(\cdot ,w) = g^0_1(w) K_{w_0}^{(1)}(\cdot, w) + \cdots + g^0_d(w) K_{w_0}^{(d)}(\cdot, w), 
$$
it follows that 
$$\big ( K^{(1)}_{w_0}(\cdot, w), \cdots , K^{(d)}_{w_0}(\cdot, w) \big ) \boldsymbol{g}^{(0)} =  \big ( K^{(1)}_{w_0}(\cdot, w), \cdots , K^{(d)}_{w_0}(\cdot, w) \big ) G(w)^\dagger \boldsymbol{s}^{(w)}.$$ 
By the uniqueness statement at $w$ of the vectors 
$\{ K^{(1)}_w(\cdot , u), \ldots ,K^{(e)}_w(\cdot , u)\}$ of Theorem \ref{sourcethm}, we conclude that 
$$ \big ( K^{(1)}_{w}(\cdot, w), \cdots , K^{(e)}_{w}(\cdot, w) \big ) = \big ( K^{(1)}_{w_0}(\cdot, w), \cdots , K^{(d)}_{w_0}(\cdot, w) \big ) G(w)^\dagger.$$ 
Taking inner product of the vectors on both sides of this equation verifies the assertion of the Theorem.  
\end{proof}

For any fixed but arbitrary $w_0\in \mathcal Z$, let $d$ be the minimal set of generators $\overline{\mathcal S}_{w_0}$. Also, we have anti-holomorphic functions 
$K^{(i)}_{w_0}(\cdot, w)$, $1\leq i \leq d$, defined on some open neighbourhood $U_{w_0}$ of $w_0$. The real analytic function  $H(w) = \big ( \!\! \big ( H_{ij}(w) \big ) \!\! \big ): = \big ( \!\! \big ( \big \langle K^{(i)}_{w_0}(\cdot, w) , K^{(j)}_{w_0}(\cdot, w) \big \rangle \big ) \!\! \big )$ defines a hermitian structure on the locally free sheaf $\mathcal O(U_{w_0})^d$. Set $\mathcal P_{w_0}$ to be the holomorphic hermitian vector bundle of rank $d$ ($d$ depends on $w_0$) determined by the pair $(\mathcal O(U_{w_0})^d, H(w))$. 
Theorem 1.10 of \cite{BMP} says  that if two analytic sheaves  $\mathcal S$ and $\mathcal {\tilde S}$ are isomorphic (that is, there is a unitary module map between the modules $\mathcal M$ an $\tilde M$ they are associated with), then the holomorphic hermitian vector bundles $\mathcal P_{w_0}$ and $\tilde{\mathcal P}_{w_0}$ are equivalent (by shrinking $U_{w_0}$ and $\tilde{U}_{w_0}$) via an isometric bundle map for each fixed but arbitrary $w_0\in \mathcal Z$ on a common open neighbourhood $U^\prime_{w_0}$ of $w_0$. If $w$ is outside the zero set, then there is an open neighbourhood $U$ of $w$ in which the dimension of the stalk $\overbar{\mathcal S}_w$ is $1$ and the eigenvectors $K(\cdot, w)$ define a holomorphic hermitian bundle $E(U)$ on the open set $U$, this is the Cowen-Douglas bundle of $\mathcal M$ on $U$. Moreover, if the two analytic sheaves $\mathcal S$ and $\mathcal {\tilde S}$ are isomorphic, then these holomorphic hermitian bundles are (locally) equivalent. 

\begin{defn}
If the vector bundles $\mathcal P_w$, $E(U)$  (determined by $K$) are equivalent to the corresponding vector bundles $\tilde{\mathcal P}_w$, 
$\tilde{E}(U)$ (determined by $\tilde{K}$) for all $w\in \mathcal Z$ and some open $U \subseteq \Omega\setminus \mathcal Z$, then  we say that the hermitian sheaves $\mathcal S$ and $\mathcal {\tilde S}$ are equivalent. \end{defn}

In particular, for each $w\in\mathcal Z$, the curvature of the determinant bundle $\det(\mathcal P_w)$ induced by the hermitian structure $\mathbb H(w)$ defined on $U_{w}$ is an invariant for the analytic Hilbert module $\mathcal M$, or equivalently, that of the holomorphic hermitian sheaf $\mathcal S(K)$. Furthermore, if we assume that the zero variety $\mathcal Z$ is a complex manifold,  then we can strengthen what we have just said, see Corollary \ref{cor:4.6}.  In this case, the dimension of $\overbar{\mathcal S}_w$ is a constant, say $t > 1$, see Theorem \ref{mainthm}. Also, for any $w\in \mathcal Z$, fixing a minimal set of generators $\{p_1, \ldots, p_t\}$, there exists a uniquely determined  holomorphic frame $\{K^1, \ldots , K^t\}$ on some open neighbourhood $\Omega_w$ of $w$, again, see Theorem \ref{mainthm}. Therefore, the locally-free sheaf $\mathcal O(\Omega_w)^t$ with the holomorphic frame $\{K^1, \ldots , K^t\}$ defines a smooth hermitian structure $\mathbb H$ for all $w$ in $\Omega_w \cap \mathcal Z$ \textit{simultaneously}. As before, the curvature of the determinant bundle induced by $\det(\mathcal P)$ on the open set $\Omega_w$ is an invariant for the analytic Hilbert module $\mathcal M$. 
Summarizing this discussion, we list useful invariants for the holomorphic hermitian sheaf $\mathcal S$ associated to a reproducing kernel Hilbert space. 
\begin{thm}
The curvature of the determinant bundles $\det(\mathcal P_w)$, $w\in \mathcal Z$, are invariants for the holomorphic hermitian sheaf $\mathcal S$. 
\end{thm}

In what follows, we rework the $(\lambda, \mu)$ examples of \cite[4.1, p. 194]{BMP} using the determinant bundle and arrive at the same conclusion.

\begin{ex}
 For any pair $\lambda, \mu >0$, let $\mathcal H^{(\lambda,\mu)}$ be the Hilbert module of wrigted Bergman spaces on the bidisc $\mathbb D^2$ determined by the kernel function $\tfrac{1}{(1-\bar{w}z)^\lambda}\tfrac{1}{(1-\bar{w}z)^\mu}$. Let $\mathcal M \subseteq \mathcal H^{(\lambda, \mu)}$ be the submodule of functions vanishing at $\{(0,0)\}$.  The hermitian metric for the sheaf $\mathcal S^\mathcal M$ in a neighbourhood of the point $\{(0,0)\}$ was calculated 
in \cite{BMP}. Taking the trace of the curvature of this metric from the bottom of page 29 in \cite{BMP}, we see that the coefficients of $dz_1 \wedge d \bar{z}_1$ and $dz_2 \wedge d \bar{z}_2$ of the curvature $\mathcal K^{(\lambda,\mu)}(z)_{|z=(0,0)}$ for the determinant bundle $\mathcal P_{(0,0)}$ are given by the formula: 
$\tfrac{\lambda+1}{2} + \tfrac{\lambda\mu^2}{(\lambda+\mu)^2},$   $\tfrac{\mu+1}{2}+ \tfrac{\lambda^2\mu}{(\lambda+\mu)^2},$
respectively. Equating these coeffcients of $\mathcal K^{(\lambda,\mu)}(z)_{|z=(0,0)}$ and 
$\mathcal K^{(\lambda^\prime,\mu^\prime)}(z)_{|z=(0,0)}$, we obtain the pair of equations:
\begin{equation}\label{eqn:3.1}
\begin{aligned}
 \tfrac{\lambda+1}{2} + \tfrac{\lambda\mu^2}{(\lambda+\mu)^2} &=    
 \tfrac{\lambda^\prime +1}{2} + \tfrac{\lambda^\prime {\mu^\prime}^2}{(\lambda^\prime +\mu^\prime)^2} \\
    \tfrac{\mu+1}{2} + \tfrac{\lambda^2\mu}{(\lambda+\mu)^2} &=    
 \tfrac{\mu^\prime +1}{2} + \tfrac{{\lambda^\prime}^2 \mu^\prime}{(\lambda^\prime +\mu^\prime)^2} \\
 \end{aligned}
 \end{equation}
 Set $k = \frac{\lambda}{\mu}$ and $k' = \frac{\lambda'}{\mu'}$. Dividing the first equation by the second in \eqref{eqn:3.1}, we have that
    $$
    \frac{k\big \{\frac{1}{2} + \frac{1}{(1 + k)^2} \big\}}{\frac{1}{2} + \frac{k^2}{(1 + k)^2}} = \frac{k' \big \{\frac{1}{2} + \frac{1}{(1 + k')^2} \big \}}{\frac{1}{2} + \frac{{k'}^2}{(1 + k')^2}} = \alpha \mbox{~(say)},
    $$
    which simplifies to the equation 
    $k(1 + k)^2 + 2k = \alpha(1 + k)^2 + 2\alpha k^2$. This shows $k$ is a root of the equation 
    \begin{equation} \label{eqn:3.2}
    x^3 - (3\alpha - 2)x^2 - (2\alpha - 3)x - \alpha = 0.
    \end{equation} 
    We claim that this equation has only one positive root for every positive $\alpha$ and hence $k = k'$. From \eqref{eqn:3.1}, it follows that $\lambda = \lambda'$ and $\mu = \mu'$. Any cubic equation can have either one real root and two complex roots or three real roots. In the first case, we are done as product of the roots is $\alpha$. In case it has 3 real roots, since the product of the roots is positive, either there is  one positive root and two negative roots or all the three roots are positive. In the first case, discarding the two negative roots, we are done again. We claim that the second  case does not occur.  If $k_i,\, i = 1,2,3$ are positive roots of the equation \eqref{eqn:3.2}, then 
    $\sum k_i = 3\alpha -2$, $\sum_{i<j}k_ik_j = 3-2\alpha$ and $k_1k_2k_3 = \alpha$ and it follows that $\alpha$ lies in $[2/3, 3/2]$. However, the identity $\sum_{i\leq j} k_ik_j = 1/2\big \{\sum_{i<j}(k_i + k_j)^2\big \}$ gives that $\alpha$ must lie in $(1, 3/2]$. Now observing that $1/k$ is also a root of $x^3 - (3\beta - 2)x^2 - (2\beta - 3)x - \beta = 0$ with $\beta = 1/\alpha$, and using the same identity, we arrive at a contradiction that $\alpha$ can not be greater than $1$.
\end{ex}

\section{Finding a holomorphic frame in an open subset of the zero set}

Let $\Omega$ be a bounded domain in $\mathbb{C}^{m}$ and $\mathcal{M}$ be a Hilbert module over the polynomial ring $\mathbb{C}[z_{1},\ldots,z_{m}]$, in the class $\mathfrak{B}_{1}(\Omega)$. We construct a sheaf $\mathcal{S}^{\mathcal{M}}$ for the Hilbert module $\mathcal{M}$ as follows:
$$\mathcal{S}^{\mathcal{M}}(U)=\Big\{\sum_{i=1}^{n}(f_{i}|_{U})g_{i}:
f_{i}\in\mathcal{M},g_{i}\in\mathcal{O}(U),n\in\mathbb{N}\Big\},~U~\text{open in}~\Omega$$
or equivalently,
$$\mathcal{S}^{\mathcal{M}}_{w}=\big\{(f_{1})_{w}\mathcal{O}_{w}+\cdots+(f_{n})_{w}\mathcal{O}_{w}:
f_{1},\ldots,f_{n}\in\mathcal{M},n\in\mathbb{N}\big\},~w\in\Omega.$$
Clearly, $\mathcal{S}^{\mathcal{M}}$ is a subsheaf of the sheaf of holomorphic functions $\mathcal{O}_{\Omega}$. From
\cite[Proposition 2.1]{SB}, it follows that $\mathcal{S}^{\mathcal{M}}$ is coherent. In particular, for each fixed $w\in\Omega$, $\mathcal{S}^{\mathcal{M}}_{w}$ is generated by finitely many elements from $\mathcal{O}_{w}$. 

Now, suppose that $\mathcal{H}$ is an analytic Hilbert module and $\mathcal{M}$ is the closure of a polynomial ideal $\mathcal{I}$ in $\mathcal{H}$ generated by $\{p_{1},\ldots,p_{t}\}$. For each $w\in\Omega$, from \cite[Lemma 2.3.2]{chenguo}, it follows that  $\{(p_{1})_{w},\ldots,(p_{t})_{w}\}$ generates the stalk $\mathcal{S}^{\mathcal{M}}_{w}$. In the following lemma we provide a sufficient condition for the minimality of such a generating set. We  let $Z(p_1, \ldots, p_t)$
denote the common zero set of the polynomials $p_1, \ldots , p_t$.
\begin{lem}\label{mainlem}
If $V(\mathcal{M}):=Z(p_{1},\ldots,p_{t})\cap\Omega$ is a submanifold of codimension $t$, then
\begin{enumerate}
\item[\rm a)] $\{p_{1},\ldots,p_{t}\}$ is a minimal set of generators for  $\mathcal{I}$ and
\item[\rm b)] for each $w\in V(\mathcal{M})$, $(p_{1})_{w},\ldots,(p_{t})_{w}$ is a minimal generator of $\mathcal{S}^{\mathcal{M}}_{w}$.
\end{enumerate}
\end{lem}
\begin{proof}
Assume that $p_{t}=q_{1}p_{1}+\cdots+q_{t-1}p_{t-1}$, for $q_{1},\ldots,q_{t-1}\in\mathbb{C}[z_{1},\ldots,z_{m}]$. Then we have
$V(\mathcal{M})=Z(p_{1},\ldots,p_{t-1})\cap\Omega$. Since $p_{1}|_{\Omega},\ldots,p_{t-1}|_{\Omega}$ are $t-1$ holomorphic functions from $\Omega$ to $\mathbb{C}$, applying \cite[Section 3.5]{EMC} we obtain that $\mbox{codim}(V(\mathcal{M}))$ is at most $t-1$ which contradicts the hypothesis of the lemma. This proves part a).

To prove part b), assume that  there exists a point $w_{0}\in V(\mathcal{M})$ such that
$$(p_{t})_{w_{0}}=(a_{1})_{w_{0}}(p_{1})_{w_{0}}+\cdots+(a_{t-1})_{w_{0}}(p_{t-1})_{w_{0}},$$
where $a_{1},\ldots,a_{t-1}$ are holomorphic functions defined on some neighbourhood $N_{w_{0}}$ of $w_{0}$ in $\Omega$. Going to a smaller neighbourhood if necessary, we have $p_{t}=a_{1}p_{1}+\cdots+a_{t-1}p_{t-1}$ on $N_{w_{0}}$. As a result, $V(\mathcal{M})\cap N_{w_{0}}=Z(p_{1},\ldots,p_{t-1})\cap N_{w_{0}}$. Now, since $p_{1}|_{N_{w_{0}}},\ldots,p_{t-1}|_{N_{w_{0}}}$ are $t-1$ holomorphic functions on $N_{w_{0}}$ and $Z(p_{1}|_{N_{w_{0}}},\ldots,p_{t-1}|_{N_{w_{0}}})=V(\mathcal{M})\cap N_{w_{0}}$ is a submanifold, from \cite[Section 3.5]{EMC} it follows that $\mbox{codim}(V(\mathcal{M})\cap N_{w_{0}})$ is at most $t-1$. Again, this is a contradiction to the hypothesis saying $\mbox{codim}(V(\mathcal{M})\cap N_{w_{0}})=\mbox{codim}(V(\mathcal{M}))=t$.
\end{proof}

\begin{rem}
Suppose $Z(p_{1},\ldots,p_{t})$ is a complete intersection, that is, the tuple $(p_{1},\ldots,p_{t}):\Omega\to\mathbb{C}^{t}$ is a submersion at every point $w\in V(\mathcal{M})$. Then $V(\mathcal{M})$ is a submanifold of codimension $t$. In this case, we can give a more direct proof of Lemma \ref{mainlem} as follows.

If there exists a point $w_{0}\in V(\mathcal{M})$ such that $p_{t}=a_{1}p_{1}+\cdots+a_{t-1}p_{t-1}$ on $N_{w_{0}}$, then, for $j=1,\ldots,m,$
\[\begin{split}
\frac{\partial p_{t}}{\partial w_{j}}(w_{0})&=\frac{\partial(a_{1}p_{1})}{\partial w_{j}}(w_{0})+\cdots+\frac{\partial(a_{t-1}p_{t-1})}{\partial w_{j}}(w_{0})\\
&=a_{1}(w_{0})\frac{\partial p_{1}}{\partial w_{j}}(w_{0})+\cdots+a_{t-1}(w_{0})\frac{\partial p_{t-1}}{\partial w_{j}}(w_{0}).
\end{split}\]
As a result, the matrix
$$\begin{pmatrix}\frac{\partial p_{1}}{\partial w_{1}}(w_{0}) & \ldots & \frac{\partial p_{1}}{\partial w_{t}}(w_{0})\\
\vdots & & \vdots\\
\frac{\partial p_{t}}{\partial w_{1}}(w_{0}) & \ldots & \frac{\partial p_{t}}{\partial w_{t}}(w_{0})
\end{pmatrix}=\begin{pmatrix}\frac{\partial p_{1}}{\partial w_{1}}(w_{0}) & \ldots & \frac{\partial p_{1}}{\partial w_{t}}(w_{0})\\ \vdots & & \vdots\\
\frac{\partial p_{t-1}}{\partial w_{1}}(w_{0}) & \ldots & \frac{\partial p_{t-1}}{\partial w_{t}}(w_{0})\\
\sum_{i=1}^{t-1}a_{i}(w_{0})\frac{\partial p_{i}}{\partial w_{1}}(w_{0}) & \ldots & \sum_{i=1}^{t-1}a_{i}(w_{0})
\frac{\partial p_{i}}{\partial w_{t}}(w_{0})
\end{pmatrix}$$
has rank at most $t-1$. This contradicts the fact that $Z(p_{1},\ldots,p_{t})$ is a complete intersection at $w_{0}$.
\end{rem}

The following Theorem is a generalization of \cite[Theorem 1.5]{BMP} reproduced earlier in this paper as Theorem \ref{sourcethm}.
\begin{thm}\label{mainthm}
Let $\mathcal{H}\subseteq \mathcal{O}(\Omega)$ be an analytic Hilbert module for some bounded domain $\Omega$  in $\mathbb{C}^{m}$. Also, let $\mathcal{M}$ be a submodule of $\mathcal{H}$ of the form
$[\mathcal{I}]$, that is, $\mathcal M$ is the completion of some polynomial ideal  $\mathcal I\subseteq \mathcal H$ with generators $p_{1},\ldots,p_{t}$. Furthermore, assume that $V(\mathcal{M})$ is a submanifold of codimension $t$. Then there exist anti-holomorphic maps $F_{1},\ldots,F_{t}:V(\mathcal{M})\to \mathcal{M}$ such that we have the following.

\begin{enumerate}
\item[\rm1.] For each $w\in V(\mathcal{M})$, there exists a neighbourhood $\Omega_{w}$ of $w$ in $\Omega$, anti-holomorphic maps
$F^{1}_{\Omega_{w}},\ldots,F^{t}_{\Omega_{w}}:\Omega_{w}\to\mathcal{M}$ with the properties listed below.
\begin{enumerate}
\item[\rm a)] $F^{j}_{\Omega_{w}}(v)=F_{j}(v)$, for all $v\in V(\mathcal{M})\cap\Omega_{w},j\in\{1,\ldots,t\}$.
\item[\rm b)] $k_u = \sum_{j=1}^{t}\overline{p_{j}(u)}F^{j}_{\Omega_{w}}(u)$, for all $u\in\Omega_{w}$,  where $k_w:=K(\cdot,w)$, $w\in \Omega$, with $K$ being the reproducing kernel of the submodule $\mathcal M$.
 \item[\rm c)] $\{F^{1}_{\Omega_{w}}(u),\ldots,F^{t}_{\Omega_{w}}(u)\}$ is a linearly independent set for each $u\in\Omega_{w}$.
\end{enumerate}
\item[\rm2.] The set $\{F_{1},\ldots,F_{t}\}$ is uniquely determined by $\{p_{1},\ldots,p_{t}\}$, that is, if $G_{1},\ldots,G_{t}$ is another collection of anti-holomorphic maps from $V(\mathcal{M})$ to $\mathcal{M}$ satisfying \mbox{\rm1.\,a)} and \mbox{\rm1.\,b)}, then $G_{j}=F_{j}$, $1\leq j\leq t$.
\item[\rm3.] $M_{p}^{*}F_{j}(v)=\overline{p(v)}F_{j}(v)$, for all $j=1,\ldots,t$, $v\in V(\mathcal{M})$, where $M_{p}$ is the multiplication by the polynomial $p$.
\item[\rm4.] For each $v\in V(\mathcal{M})$, the linear span of the set of vectors $\{F_{1}(v),\ldots,F_{t}(v)\}$ in $\mathcal{M}$ is the joint kernel of $\mathcal{M}$ at $v$ and hence is independent of the choice of generators $p_{1},\ldots,p_{t}$.
\end{enumerate}
\end{thm}

\begin{proof}
Pick an arbitrary point $w\in V(\mathcal{M})$. From Lemma \ref{mainlem}, $(p_{1})_{w},\ldots,(p_{t})_{w}$ is a minimal set of generator of $\mathcal{S}^
{\mathcal{M}}_{w}$. Consequently, there exists a neighbourhood $\Omega_{w}$ of $w$ in $\Omega$ such that for all $u\in\Omega_{w}$, $K(\cdot,u)=
\sum_{j=1}^{t}\overline{p_{j}(u)}F^{j}_{\Omega_{w}}(u)$, where $F^{1}_{\Omega_{w}},\ldots,F^{t}_{\Omega_{w}}$ are anti-holomorphic maps from $\Omega_{w}$ to $\mathcal{M}$ satisfying conditions (ii) to (v) of 
Theorem \ref{sourcethm}

Take $w_{1},w_{2}\in V(\mathcal{M})$ such that $\Omega_{w_{1}}\cap\Omega_{w_{2}}\cap V(\mathcal{M})$ is non-empty. Now, for each $u\in
\Omega_{w_{1}}\cap\Omega_{w_{2}}$ we have
$$K(\cdot,u)=\sum_{j=1}^{t}\overline{p_{j}(u)}F^{j}_{\Omega_{w_{1}}}(u)\quad\text{and}\quad
K(\cdot,u)=\sum_{j=1}^{t}\overline{p_{j}(u)}\tilde{F}^{j}_{\Omega_{w_{2}}}(u).$$
This implies
$$\sum_{j=1}^{t}\overline{p_{j}(u)}\big(F^{j}_{\Omega_{w_{1}}}(u)-\tilde{F}^{j}_{\Omega_{w_{2}}}(u)\big)=0.$$
For each $u\in\Omega_{w_{1}}\cap\Omega_{w_{2}}$, $1\leq j\leq t$, set $\alpha_{j}(u)=\overline{\big(F^{j}_{\Omega_{w_{1}}}(u)-\tilde{F}^{j}_{\Omega_{w_{2}}}(u)\big)}$ and note that $\sum_{j=1}^{t}p_{j}(u)\alpha_{j}(u)=0$. Now, fix an arbitrary $v\in \Omega_{w_{1}}\cap\Omega_{w_{2}}\cap V(\mathcal{M})$ and assume that $\alpha_{1}(v)\neq 0$. This gives $\sum_{j=1}^{t}(p_{j})_{v}(\alpha_{j}) _{v}=0$ in $\mathcal{O}_{\mathbb{C}^{m},v}$ and $(\alpha_{1})_{v}$ is an unit in $\mathcal{O}_{\mathbb{C}^{m},v}$. Consequently, $(p_{1})_{v}=-\sum_{j=2}^{t}\big((\alpha_{1})_{v}^{-1}(\alpha_{j})_{v}\big)(p_{j})_{v}$ which says that $\{(p_{1})_{v},\ldots,(p_{t})_{v}\}$ is not a minimal set of generators of $\mathcal{S}^{\mathcal{M}}_{v}$ contradicting Lemma \ref{mainlem}. Thus,
\begin{equation}\label{eqn1}
\alpha_{j}(v)=0\Leftrightarrow F^{j}_{\Omega_{w_{1}}}(v)=\tilde{F}^{j}_{\Omega_{w_{2}}}(v),\forall v\in
\Omega_{w_{1}}\cap\Omega_{w_{2}}\cap V(\mathcal{M}),1\leq j\leq t.
\end{equation}
Since $\{\Omega_{w}\cap V(\mathcal{M})\}_{w\in V(\mathcal{M})}$ is an open cover of $V(\mathcal{M})$, for each $j=1,\ldots,t$, we define
$F_{j}:V(\mathcal{M})\to\mathcal{M}$ as follows:
$$F_{j}|_{\Omega_{w}\bigcap V(\mathcal{M})}(v):=F^{j}_{\Omega_{w}}(v),\forall v\in\Omega_{w}\cap V(\mathcal{M}).$$
From \ref{eqn1} it follows that for each $j=1,\ldots,t$, $F_{j}$ is a well-defined, anti-holomorphic map satisfying 1.a),1.b) and 1.c).

\bigskip
To prove 2., assume that for each $w\in V(\mathcal{M})$, there exist a neighbourhood $N_{w}$ of $w$ in $\Omega$, anti-holomorphic maps $G^{1}_{N_{w}},\ldots,
G^{t}_{N_{w}}:N_{w}\to\mathcal{M}$ such that a) $G^{j}_{N_{w}}(v)=G_{j}(v)$, $\forall v\in N_{w}\cap V(\mathcal{M})$ and b) $K(\cdot,u)=
\sum_{j=1}^{t}\overline{p_{j}(u)}G^{j}_{N_{w}}(u)$, $\forall u\in N_{w}$. This gives
$$\sum_{j=1}^{t}\overline{p_{j}(u)}\big(F^{j}_{\Omega_{w}}(u)-G^{j}_{N_{w}}(u)\big)=0,\forall u\in\Omega_{w}\cap N_{w}.$$
Following similar arguments as given above, for each $v\in\Omega_{w}\cap N_{w}\cap V(\mathcal{M})$, we have $G_{j}(v)=
G^{j}_{N_{w}}(v)=F^{j}_{\Omega_{w}}(v)=F_{j}(v)$. In particular, $F_{j}(w)=G_{j}(w)$, for all $w\in V(\mathcal{M})$.

\bigskip
The proof of Part 3 is straightforward from condition (v) of 
Theorem \ref{sourcethm}
and from the observation that for each $w\in V(\mathcal{M})$, $F_{j}(w)=F^{j}_{\Omega_{w}}(w)$. From the same observation we obtain that $\text{Span}\{F_{1}(w),\ldots,F_{t}(w)\}$ is a subspace of $\cap_{p\in\mathbb{C}[\underline{z}]}\text{Ker}(M_{p}^{*}-
\overline{p(w)})$, for each $w\in V(\mathcal{M})$. Now, from 1.a) and 1.c) it follows that the dimension of this subspace is at least $t$. On the other hand, from \cite[Lemma 5.11]{DP} it follows that $\dim(\mathcal{M}\otimes_{\mathbb{C}[\underline{z}]}\mathbb{C}_{w})=
\dim\bigl(\cap_{p\in\mathbb{C}[\underline{z}]}\ker(M_{p}-p(w))^{*}\otimes\mathbb{C}\bigr)\leq t$. So, for each $w\in V(\mathcal{M})$, $\mbox{Span}\{F_{1}(w),\ldots,F_{t}(w)\}=\cap_{p\in\mathbb{C}[\underline{z}]}\ker(M_{p}^{*}-\overline{p(w)})$ proving Part 4.
\end{proof}

\begin{cor}\label{cor1}
Let $\mathcal{H}\subseteq \mathcal{O}(\Omega)$ be an analytic Hilbert module for some bounded domain $\Omega$  in $\mathbb{C}^{m}$. Also, let $\mathcal{M}$ be a submodule of $\mathcal{H}$ of the form
$[\mathcal{I}]$ and  $p_{1},\ldots,p_{t}$ be the generators of $\mathcal I$. Assume that $V(\mathcal{M})$ is a submanifold of codimension $t$. Then 
$$\dim \Big (\bigcap_{i=1}^{m}\ker(M_{z_{j}}-w_{j})^{*} \Big )=
\begin{cases}
1 & \mbox{\rm for}~w\notin V(\mathcal{I})\cap\Omega\\
\mbox{\rm codimension of}~V(\mathcal{I}) & \mbox{\rm for}~w\in V(\mathcal{I})\cap\Omega.
\end{cases}
$$
\end{cor}
\begin{proof}
For each $w\in V(\mathcal{M})=V(\mathcal{I})\cap\Omega$, from the proof of Part 4., Theorem \ref{mainthm}, we infer that 
$$\dim \Bigl(\bigcap_{i=1}^{m}\ker(M_{z_{j}}-w_{j})^{*}\Bigr)=\dim\Bigl
(\bigcap_{p\in\mathbb{C}[\underline{z}]}\ker(M_{p}^{*}-\overline{p(w)})\Bigr)=t.$$
Moreover, clearly $\mathcal{M}\in\mathfrak{B}_{1}(\Omega)$. Therefore, from \cite[Lemma 1.11]{SB}, it follows that $\mathcal{M}$ is locally free on $(\Omega\setminus
V(\mathcal{M}))^*$ with the result that  $\dim\big(\cap_{i=1}^{m}\ker(M_{z_{j}}-w_{j})^{*}\big)=1$, for $w\notin V(\mathcal{M})$.
\end{proof}

Observe that we do not need $\mathcal{I}$ to be a prime ideal to prove Corollary \ref{cor1}. This enables us to consider examples that satisfy the conjecture of Douglas, Misra and Varughese \cite{DMV} but does not follow from \cite[Theorem 2.3]{DG}. We will discuss an explicit example below to demonstrate this.
In what follows, we let $< \{p_1, \ldots , p_t\} >$ denote the ideal generated by the polynomials $p_1, \ldots , p_t$.
Consider the ideal $\mathcal{I}=<\{z_{1}z_{2},z_{1}-z_{2}\}>$ in $\mathbb{C}[z_{1},\ldots,z_{m}]$ and define $\mathcal{M}=[\mathcal{I}]$ in $\mathcal{H}=H^{2}
(\mathbb{D}^{m})$, where $m\geq 2$. If $z\in V(\mathcal{I})$, then $z_{1}z_{2}=0$ and $z_{1}-z_{2}=0$.  Thus, $V(\mathcal{I})\subseteq\{z_{1}=z_{2}=0\}$.  Furthermore, observe that $$\mathcal{I}\subseteq<\{z_{1},z_{1}-z_{2}\}>=<\{z_{1},z_{2}\}>.$$
Note that for two ideals $\mathcal{I}_{1},\mathcal{I}_{2}$ in $\mathbb{C}[z_{1},\ldots,z_{m}]$, if $\mathcal{I}_{1}\subseteq\mathcal{I}_{2}$, then
$V(\mathcal{I}_{2})\subseteq V(\mathcal{I}_{1.})$.
As a result, we obtain $V(\mathcal{I})=\{(z_1, \ldots, z_m):z_{1}=z_{2}=0\}$ which implies that $V(\mathcal{M})=\{(z_1, \ldots,z_m)\in \mathbb D^m:z_{1}=z_{2}=0\}$. So, by Corollary \ref{cor1} we have
$$\dim(\mathcal{M}\otimes_{\mathbb{C}[\underline{z}]}\mathbb{C}_{w})=\begin{cases} 1 & \mbox{\rm for}~w\notin V(\mathcal{I})\cap\mathbb{D}^{m},\\
2 & \mbox{\rm for}~w\in V(\mathcal{I}) \cap\mathbb{D}^{m}.\end{cases}$$
Now, we claim that $\mathcal{I}$ is not prime. To see this, it is enough to show that neither $z_{1}$ nor $z_{2}$ belongs to $\mathcal{I}$. Assume   $z_{1}\in\mathcal{I}$. Then, there exist $p_{1},p_{2}\in\mathbb{C}[z_{1},\ldots,z_{m}]$ such that
$$z_{1}=p_{1}z_{1}z_{2}+p_{2}(z_{1}-z_{2})$$
which implies
$$z_{1}(1-p_{2}-p_{1}z_{2})=-p_{2}z_{2}.$$
This means $z_{1}$ divides $p_{2}z_{2}$. But $z_{1}$ is a prime element of $\mathbb{C}[z_{1},\ldots,z_{m}]$ and $z_{1}$ does not divide $z_{2}$. So, $z_{1}$ divides $p_{2}$ and we will write $p_{2}=qz_{1}$, for some $q\in \mathbb{C}[z_{1},\ldots,z_{m}]$. Thus, we have
$$z_{1}=p_{1}z_{1}z_{2}+qz_{1}(z_{1}-z_{2}).$$
Finally, dividing both sides by $z_{1}$ we have
$$1=p_{1}z_{2}+q(z_{1}-z_{2}).$$
This is a contradiction because the right hand side vanishes at the origin whereas the left hand side does not. This proves $z_{1}\notin\mathcal{I}$. Following similar arguments one can show that $z_{2}\notin\mathcal{I}$.

\section{The vector bundle associated to the joint kernel and its curvature}

From Theorem \ref{mainthm} we obtain a rank $t$, trivial, anti-holomorphic bundle $E_{\mathcal{M}}$ on $V(\mathcal{M})=V(\mathcal{I})\cap\Omega$ corresponding to the set $\{p_{1},\ldots,p_{t}\}$ given by
$$E_{\mathcal{M}}=\bigsqcup_{w\in V(\mathcal{M})}<\{F_{1}(w),\ldots,F_{t}(w)\}>.$$
Since, for each $w\in V(\mathcal{M})$, $(E_{\mathcal{M}})_{w}$ is a subspace of $\mathcal{M}$, we can give a hermitian structure on $E_{\mathcal{M}}$ which is canonically induced by the inner product of $\mathcal{M}$. This observation leads to the following theorem.
\begin{thm}\label{thm2}
Let $\Omega$ be a bounded domain in $\mathbb{C}^{m}$ and $\mathcal{I}$ be a polynomial ideal with generators $p_{1},\ldots,p_{t}$. Also, let $\mathcal{H},\mathcal{H'}$ be analytic Hilbert modules in $\mathcal{O}(\Omega)$ and $\mathcal{M},\mathcal{M'}$ be the closure of $\mathcal{I}$ in $\mathcal{H},\mathcal{H'}$, respectively, with the property that $\mbox{\rm codim}V(\mathcal{M})=\mbox{\rm codim}V(\mathcal{M'})=t$. If the modules $\mathcal{M},\mathcal{M'}$ are "unitarily" equivalent, then we have the following:
\begin{enumerate}
\item[\rm (a)]$E_{\mathcal{M}}$ is equivalent to $E_{\mathcal{M'}}$, where $E_{\mathcal{M}},E_{\mathcal{M'}}$ are two bundles on $V(\mathcal{M})=V(\mathcal{M'})=V(\mathcal{I})\cap\Omega$ obtained from Theorem \mbox{\rm\ref{mainthm}} and the discussion above.
\item[\rm (b)]If $F_{\mathcal{M}}:=\{F_{1},\ldots,F_{t}\},F_{\mathcal{M'}}:=\{F_{1}^{'},\ldots,F_{t}^{'}\}$ are the global frames of $E_{\mathcal{M}},E_{\mathcal{M'}}$ respectively that are obtained from applying Theorem \mbox{\rm\ref{mainthm}} on $\mathcal{M},\mathcal{M'}$ with respect to the set $\{p_{1},\ldots,p_{t}\}$, then
    $$\mathcal{K}_{E_{\mathcal{M}}}(F_{\mathcal{M}})=\mathcal{K}_{E_{\mathcal{M'}}}(F_{\mathcal{M'}}).$$
    Here $\mathcal{K}_{E_{\mathcal{M}}}(F_{\mathcal{M}}),\mathcal{K}_{E_{\mathcal{M'}}}(F_{\mathcal{M'}})$ are the curvature matrices of $E_{\mathcal{M}},E_{\mathcal{M'}}$ with respect to the frames $F_{\mathcal{M}},F_{\mathcal{M'}}$, respectively.
\end{enumerate}
\end{thm}
\begin{proof}
From Theorem \ref{mainthm}, it follows that for each $w\in V(\mathcal{I})\cap\Omega$, $(E_{\mathcal{M}})_{w}=\cap_{i=1}^{m}\ker(M_{z_{i}}-w_{i})^{*}$ and
$(E_{\mathcal{M'}})_{w}=\cap_{i=1}^{m}\ker(M^{'}_{z_{i}}-w_{i})^{*}$, where $M_{z_{i}},M^{'}_{z_{i}}$ are pointwise multiplication by $z_{i}$ on $\mathcal{M},\mathcal{M'}$ respectively. If $L:\mathcal{M}\to\mathcal{M'}$ is an unitary module map, then, for any $g\in\cap_{i=1}^{m}
\ker(M_{z_{i}}-w_{i})^{*}$, $(M^{'}_{z_{i}})^{*}(Lg)=L(M_{z_{i}}^{*}g)=L(\bar{w}_{i}g)=\bar{w}_{i}(Lg)$. Thus,
$$L\Big(\cap_{i=1}^{m}\ker(M_{z_{i}}-w_{i})^{*}\Big)\subseteq\cap_{i=1}^{m}\ker(M^{'}_{z_{i}}-w_{i})^{*}.$$
Similarly, it can be shown that
$$L^{-1}\Big(\cap_{i=1}^{m}\ker(M^{'}_{z_{i}}-w_{i})^{*}\Big)\subseteq\cap_{i=1}^{m}\ker(M_{z_{i}}-w_{i})^{*},$$
which is equivalent to
$$\cap_{i=1}^{m}\ker(M^{'}_{z_{i}}-w_{i})^{*}\subseteq L\Big(\cap_{i=1}^{m}\ker(M_{z_{i}}-w_{i})^{*}\Big).$$
Consequently, $L\big(\cap_{i=1}^{m}\ker(M_{z_{i}}-w_{i})^{*}\big)=\cap_{i=1}^{m}\ker(M^{'}_{z_{i}}-w_{i})^{*}$ and hence $L$ is an isometric isomorphism between $(E_{\mathcal{M}})_{w}$ and $(E_{\mathcal{M'}})_{w}$ for all $w\in V(\mathcal{I})\cap\Omega$. Moreover, there exists an anti-holomorphic map $A:V(\mathcal{I})\cap\Omega\mapsto GL(t,\mathbb{C})$ such that for each $w$, the following matrix equality is true:
\begin{equation}\label{eqn2}
\begin{bmatrix}LF_{1}(w) & \cdots & LF_{t}(w)\end{bmatrix}=\begin{bmatrix}F_{1}^{'}(w) & \cdots & F_{t}^{'}(w)\end{bmatrix}A(w).
\end{equation}
Thus, $L$ induces a bundle isomorphism between $E_{\mathcal{M}}$ and $E_{\mathcal{M'}}$ which proves part (a).

\bigskip
Next, observe that part (b) follows trivially from part (a) when $t=1$. This is because from part (a) we obtain the equality of the first Chern forms on $E_{\mathcal{M}},E_{\mathcal{M'}}$ given by $c_{1}(E_{\mathcal{M}},N_{\mathcal{M}})=c_{1}(E_{\mathcal{M'}},N_{\mathcal{M'}})$ which implies $\frac{i}{2\pi}
\mathcal{K}_{E_{\mathcal{M}}}(F_{\mathcal{M}})=\frac{i}{2\pi}\mathcal{K}_{E_{\mathcal{M'}}}(F_{\mathcal{M'}})
\Leftrightarrow \mathcal{K}_{E_{\mathcal{M}}}(F_{\mathcal{M}})=\mathcal{K}_{E_{\mathcal{M'}}}(F_{\mathcal{M'}})$. Here $N_{\mathcal{M}},N_{\mathcal{M'}}$ are the hermitian metrics on $E_{\mathcal{M}},E_{\mathcal{M'}}$ induced by the inner products of $\mathcal{M},\mathcal{M'}$, respectively.

\bigskip
For the case where $t\geq 2$, note that $\mathcal{M},\mathcal{M'}\in\mathfrak{B}_{1}(\Omega)$. As a result, following \cite[Lemma 1.11]{SB} we obtain that the reproducing kernels of $\mathcal{M},\mathcal{M'}$ are sharp on $\Omega\setminus V(\mathcal{I})$. Since $L$ is unitary, there exists a non-vanishing, holomorphic function $\phi$ on $\Omega\setminus V(\mathcal{I})$ such that
\begin{equation}\label{eqn3}
LK(\cdot,w)=\overline{\phi(w)}K'(\cdot,w),
\end{equation}
for all $w\in\Omega\setminus V(\mathcal{I})$ \cite[Theorem 3.7]{CS}, where $K,K'$ are the reproducing kernels of $\mathcal{M},\mathcal{M'}$, respectively. But $\mbox{codim}V(\mathcal{M})\geq 2$. So, by the Hartog's Extension Theorem \cite[Page 198]{SL} $\phi$ can be uniquely extended to $\Omega$ as a holomorphic function. Since an analytic function is unambiguously determined from its definition on any open set,  it  follows that Equation \eqref{eqn3} is true for all $w\in\Omega$.

Now, fix an arbitrary point $w_{0}\in V(\mathcal{I})\cap\Omega$. Then, from condition b) of Theorem \ref{mainthm} we have
\begin{equation}\label{eqn4}
K(\cdot,u)=\sum_{j=1}^{t}\overline{p_{j}(u)}F^{j}_{\Omega_{w_{0}}}(u)~\mbox{and}~
K'(\cdot,u)=\sum_{j=1}^{t}\overline{p_{j}(u)}{F'}^{j}_{\Omega_{w_{0}}}(u),
\end{equation}
for all $u\in\Omega_{w_{0}}$. Applying $L$ to the first equality of the Equation \eqref{eqn4} we obtain that
$$\overline{\phi(u)}K'(\cdot,u)=\sum_{j=1}^{t}\overline{p_{j}(u)}LF^{j}_{\Omega_{w_{0}}}(u)$$
or, equivalently,
$$K'(\cdot,u)=\sum_{j=1}^{t}\overline{p_{j}(u)}\frac{LF^{j}_{\Omega_{w_{0}}}(u)}{\overline{\phi(u)}}.$$
Since, $v\mapsto\frac{LF_{j}(v)}{\overline{\phi(v)}},v\mapsto F^{'}_{j}(v)$ are anti-holomorphic maps from $V(\mathcal{M'})$ to $\mathcal{M'}$ for each $j=1,\ldots,t$, satisfying 1.a) and 1.b) of Theorem \ref{mainthm}, by the condition 2 of the same theorem, we have $LF_{j}(v)=\overline{\phi(v)}F^{'}_{j}(v)$, for all $v\in V(\mathcal{I})\cap\Omega$. As a result, from Equation \eqref{eqn2} we have $A(v)=\overline{\phi(v)}I_{t\times t}$, for all $v\in V(\mathcal{I})\cap\Omega$. Finally,
\[\begin{split}\big(\mathcal{K}_{E_{\mathcal{M}}}(F_{\mathcal{M}})\big)(v)&=
\partial\Big(\big(N_{\mathcal{M}}(F_{\mathcal{M}})\big)^{-1}\bar\partial \big(N_{\mathcal{M}}(F_{\mathcal{M}})\big)\Big)(v)\\
&=\partial\Big(\big(N_{\mathcal{M'}}(LF_{\mathcal{M}})\big)^{-1}\bar\partial\big(N_{\mathcal{M'}}
(LF_{\mathcal{M}})\big)\Big)(v)\\
&=\big(\mathcal{K}_{E_{\mathcal{M'}}}(LF_{\mathcal{M}})\big)(v)\\
&=\big(\mathcal{K}_{E_{\mathcal{M'}}}(F_{\mathcal{M'}}A)\big)(v)\\
&=A(v)^{-1}\cdot\big(\mathcal{K}_{E_{\mathcal{M'}}}(F_{\mathcal{M'}})\big)(v)\cdot A(v)\\
&=\big(\mathcal{K}_{E_{\mathcal{M'}}}(F_{\mathcal{M'}})\big)(v),
\end{split}\]
which proves part (b) of the Theorem.
\end{proof}
The following example demonstrates the use of the curvature invariant given in Theorem \ref{thm2}(b). We first consider the case of a principal ideal and follow it up with the case of ideals that are not principal.
\begin{ex}
Let $\mathcal M = \{f\in \mathcal H^{\lambda, \mu}(\mathbb D^2): f_{|{z_1=0}} = \partial_1f_{|{z_1=0}} = \cdots = \partial_1^{p - 1} f_{|{z_1=0}} = 0\}$. By \cite[Lemma 4.14]{GM}, $\mathcal M$ is the closure $[\mathcal I]$ of the ideal $\mathcal I = <z_1^p>$ in $\mathcal H^{\lambda, \mu}(\mathbb D^2)$. On the zero set $\{(0, w_2): w_2\in \mathbb D\}$, the frame is given by
$$
F_1(z, (0, w_2)) = \frac{(\lambda)_p z_1^p}{p!(1 - z_2\bar w_2)^\mu},
$$
and hence
$$
\|F_1(z, (0, w_2))\|^2 =\frac{(\lambda)_p}{p!}\sum_{n\geq 0} \frac{(\mu)_n}{n!} |w_2|^{2n},
$$
where $(\lambda)_p = \lambda(\lambda + 1)\ldots (\lambda + p -1)$ is the Pochhammer symbol. The curvature at $0$ relative to this frame  is
\begin{equation}\label{cc}
 \partial\bar\partial \log \|F_1(z, (0, w_2))\|^2_{\{w_2 = 0\}} = \frac{\mu (\lambda)_p}{p!}.
\end{equation}

We note that $\mathcal I_k = <z_1^{k + p}>$ is contained in $\mathcal I$ for each $k\in\mathbb N$. If the completion $[\mathcal I]$ of the ideal $\mathcal I$ in $\mathcal H^{\lambda,\mu}(\mathbb D^2)$ is equivalent to the completion of $[\mathcal I]^\prime$ of the ideal $[\mathcal I]$ taken in $\mathcal H^{\lambda^\prime,\mu^\prime}(\mathbb D^2)$, then the completion $[\mathcal I_k]$ of $\mathcal I_k$ in $\mathcal H^{\lambda,\mu}(\mathbb D^2)$ is equivalent to the completion $[\mathcal I_k]^\prime$ of $\mathcal I_k$ in $\mathcal H^{\lambda^\prime,\mu^\prime}(\mathbb D^2)$. This follows from the fact that 
the kernel $K_{[\mathcal I]}$ and $K_{[\mathcal I]^\prime}$ factor:  $K_{[\mathcal I]}(z,w) = z_1^p\chi(z, w)\bar w_1^p$ and $K_{[\mathcal I]^\prime}(z,w) = z_1^p\chi^\prime(z, w)\bar w_1^p$.  Thus if $[\mathcal I]$ and $[\mathcal I]^\prime$ are equivalent, then there exists a non-zero holomorphic map $\varphi$ on $\mathbb D^2$ that induces the unitary module map between $\mathcal I]$ and $[\mathcal I]^\prime$. Following the same idea as in the  the proof of \cite[Lemma 4.1]{BM}, $[\mathcal I_k]$ and $[\mathcal I_k]^\prime$ are equivalent via the module map induced by $\varphi$. By considering the case $k=1$, from Theorem \ref{thm2} and the curvature computation in Equation (\ref{cc}), it follows that
$
\mu (\lambda)_p = \mu^\prime (\lambda^\prime)_p \mbox{~and~} \mu (\lambda)_{p + 1} = \mu^\prime (\lambda^\prime)_{p + 1}.
$
Thus, we have $\lambda = \lambda^\prime$ and $\mu = \mu^\prime$.

Consider the submodule $[\mathcal I]$ that is the completion of the ideal $\mathcal I = <z_1^{i_1}, \ldots, z_m^{i_m}>$ in $\mathcal H^{\boldsymbol{\lambda}} (\mathbb D^n)$ for some $m < n$ and   $\boldsymbol{\lambda} = (\lambda_1, \ldots, \lambda_n)$. The reproducing kernel of $\mathcal H^{\boldsymbol{\lambda}} (\mathbb D^n)$ is  $1/\prod_{i=1}^n (1 - z_i\bar w_i)^{\lambda_i}$. On the zero set $\{(0,, \ldots,0, w_{m + 1},\ldots, w_n)\in\mathbb C^n: w_i\in \mathbb D, m + 1\leq i\leq n\}$, the frame is given by
$$
F_k(z, (0,, \ldots,0, w_{m + 1},\ldots, w_n)) = \frac{(\lambda_k)_{i_k} z_1^{i_k}}{{i_k}!\prod_{i = m + 1}^n(1 - z_i\bar w_i)^{\lambda_i}},
$$
and hence
$$
\|F_k(z, (0,, \ldots,0, w_{m + 1},\ldots, w_n))\|^2 =\frac{(\lambda_k)_{i_k}}{{i_k}!\prod_{i = m + 1}^n(1 - |w_i|^2)^{\lambda_i}},
$$
for $1\leq k\leq m$. Since $\langle F_i, F_j\rangle = 0$ for $i\neq j$, the metric is a $m\times m$ diagonal matrix with diagonals $\|F_k\|^2$ as above. Thus, we just need to look at $$
\frac{\partial^2}{\partial w_i \partial \bar w_i} \log \|F_1(z, (0,, \ldots,0, w_{m + 1},\ldots, w_n))\|^2
$$ for $1\leq k\leq m$ and $m + 1\leq i\leq n$ as the mixed derivatives are zero at the origin. So the curvature computation at the origin and calculation similar to the case of the principal ideal, we have the completion $[\mathcal I]$ of the ideal $\mathcal I$ in $\mathcal H^{\boldsymbol{\lambda}} (\mathbb D^n)$ and the completion $[\mathcal I]^\prime$ in  $H^{\boldsymbol{\lambda^\prime}} (\mathbb D^n)$ are equivalent if and only if $\boldsymbol{\lambda} = \boldsymbol{\lambda^\prime}$.
\end{ex}

\section{On the relationship of curvature with the generators}Two different sets of generators for an ideal $\mathcal I$ give rise to two distinct  holomorphic frames for the holomorphic hermitian  vector bundle associated to the completion $[\mathcal I]$ in the Hilbert module $\mathcal H$. The curvature computed relative to these frames are equivalenrt via an invertible map which is explicitly computed below in the case when $\Omega$ is a bounded domain containing the zero vector in $\mathbb{C}^{m}$ and  $\{p_{1},\ldots,p_{t}\}$, $\{q_{1},\ldots,q_{t}\}$ are two sets of generators of the polynomial ideal $\mathcal{I}$ consisting of homogeneous polynomials of the same degree. Let $\mathcal{H}$ be an analytic Hilbert module in $\mathcal{O}(\Omega)$ and $\mathcal{M}$ be the closure of $\mathcal{I}$ in $\mathcal{H}$ with the property that $\mbox{codim}V(\mathcal{M})=t$. Then, by \cite[Lemma 4.7]{SB} we obtain that
$$\big\{p_{1}(\bar{D})K(\cdot,w)|_{w=0},\ldots,p_{t}(\bar{D})K(\cdot,w)|_{w=0}\big\}~\mbox{and}~
\big\{q_{1}(\bar{D})K(\cdot,w)|_{w=0},\ldots,q_{t}(\bar{D})K(\cdot,w)|_{w=0}\big\}$$
are two bases of $\ker D_{M^{*}}$. As a result, from \cite[Lemma 4.2]{SB} and \cite[Proposition 4.11]{SB} it follows that there exists a constant invertible matrix $A=(a_{ij})_{i,j=1}^{t}$ such that,
\begin{equation}\label{eqn5}
q_{j}=\sum_{i=1}^{t}a_{ij}p_{i},~1\leq j\leq t.
\end{equation}
Applying Theorem \ref{mainthm} separately with the choice of   $\{p_{1},\ldots,p_{t}\}$ and $\{q_{1},\ldots,q_{t}\}$, we obtain the sets $\{F_{1}^{p},\ldots,F_{t}^{p}\},\{F_{1}^{q},\ldots,F_{t}^{q}\}$, respectively which consist of the anti-holomorphic maps from $V(\mathcal{M})$ to $\mathcal{M}$ satisfying conditions 1) to 4) of the theorem. Now, we claim the following.
\begin{lem}
For each $w\in V(\mathcal{M})$, $\begin{bmatrix}F_{1}^{p}(w) &\ldots & F_{t}^{p}(w)\end{bmatrix}=\begin{bmatrix}F_{1}^{q}(w) &\ldots & F_{t}^{q}(w)\end{bmatrix} A^{*}$.
\end{lem}
\begin{proof}
Fix an arbitrary point $w_{0}\in V(\mathcal{M})$. Then from condition 1 of Theorem \ref{mainthm} there exist a neighbourhood $\Omega_{w_{0}}$ of $w_{0}$ in $\Omega$, anti-holomorphic maps $F^{k}_{\Omega_{w_{0}},p},F^{k}_{\Omega_{w_{0}},q}:V(\mathcal{M})\to\mathcal{M}$, $k=1,\ldots,t$ such that
\begin{enumerate}
\item[a)]$F^{k}_{\Omega_{w_{0}},p}(v)=F_{k}^{p}(v),F^{k}_{\Omega_{w_{0}},q}(v)=F_{k}^{q}(v)$, for all $v\in V(\mathcal{M})\cap\Omega_{w_{0}}$, $k=1,\ldots,t$ and
\item[b)]$K(\cdot,u)=\sum_{i=1}^{t}\overline{p_{i}(u)}F^{i}_{\Omega_{w_{0}},p}(u)$, $K(\cdot,u)=\sum_{j=1}^{t}\overline{q_{j}(u)}F^{j}_{\Omega_{w_{0}},q}(u)$, for all $u\in\Omega_{w_{0}}$.
\end{enumerate}
Now, applying Equation \eqref{eqn5} to the second equality of b) we obtain that
$$K(\cdot,u)=\sum_{i=1}^{t}\overline{p_{i}(u)}\Big(\sum_{j=1}^{t}\bar{a}_{ij}F^{j}_{\Omega_{w_{0}},q}(u)\Big),$$
for all $u\in\Omega_{w_{0}}$. Finally, observe that, for each $i\in\{1,\ldots,t\}$, $v\mapsto\sum_{j=1}^{t}\bar{a}_{ij}F_{j}^{q}(v)$ is an anti-holomorphic map from $V(\mathcal{M})\to\mathcal{M}$ satisfying conditions 1.a) and 1.b) of Theorem \ref{mainthm} with respect to the set $\{p_{1},\ldots,p_{t}\}$. So, from condition 2 of Theorem \ref{mainthm}, it follows that, for each $i=1,\ldots,t$, $v\in V(\mathcal{M})$,
$$F_{i}^{p}(v)=\sum_{j=1}^{t}\bar{a}_{ij}F_{j}^{q}(v),$$
proving the lemma.
\end{proof}
Note that each of the sets $\{F_{1}^{p},\ldots,F_{t}^{p}\}$ and $\{F_{1}^{q},\ldots,F_{t}^{q}\}$ canonically induces an anti-holomorphic frame of $E_{\mathcal{M}}$ on $V(\mathcal{M})$. If we denote the frames as $F_{\mathcal{M}}^{p}$ and $F_{\mathcal{M}}^{q}$ respectively, then we have
$$\mathcal{K}_{E_{\mathcal{M}}}(F_{\mathcal{M}}^{p})=\mathcal{K}_{E_{\mathcal{M}}}(F_{\mathcal{M}}^{q}A^{*})=
(A^{*})^{-1}\cdot
\mathcal{K}_{E_{\mathcal{M}}}(F_{\mathcal{M}}^{q})\cdot A^{*},$$
where $\mathcal{K}_{E_{\mathcal{M}}}(F_{\mathcal{M}}^{p}),\mathcal{K}_{E_{\mathcal{M}}}(F_{\mathcal{M}}^{q})$ are the curvature matrices of the bundle $E_{\mathcal{M}}$ with respect to the frames $F_{\mathcal{M}}^{p},F_{\mathcal{M}}^{q}$, respectively. Thus, we have proved the following.
\begin{prop}\label{prop1}
Let $\Omega$ be a bounded domain in $\mathbb{C}^{m}$ and $\{p_{1},\ldots,p_{t}\},\{q_{1},\ldots,q_{t}\}$ be two generators of the polynomial ideal $\mathcal{I}$ consisting of homogeneous polynomials of same degree. Also, let $\mathcal{H}\subseteq \mathcal{O}(\Omega)$ be an analytic Hilbert module and $\mathcal{M}$ be the closure of $\mathcal{I}$ in $\mathcal{H}$ with the property that $\mbox{\rm codim}V(\mathcal{M})=t$. Furthermore, assume that $F_{\mathcal{M}}^{p},
F_{\mathcal{M}}^{q}$ are the global frames of $E_{\mathcal{M}}$ obtained by applying Theorem \mbox{\rm\ref{mainthm}} on $\mathcal{M}$ with respect to the generators mentioned above. Then there exists a constant invertible matrix $A$ such that
$$\mathcal{K}_{E_{\mathcal{M}}}(F_{\mathcal{M}}^{p})=(A^{*})^{-1}\cdot
\mathcal{K}_{E_{\mathcal{M}}}(F_{\mathcal{M}}^{q})\cdot A^{*}.$$
\end{prop}
\begin{cor} \label{cor:4.6}
Let $\Omega,\mathcal{I}$ be as above, $\mathcal{H},\mathcal{H'}\subseteq\mathcal{O}(\Omega)$ be two analytic Hilbert modules and $\mathcal{M},\mathcal{M'}$ be the closure of $\mathcal{I}$ in $\mathcal{H},\mathcal{H'}$ respectively with $\mbox{\rm codim}V(\mathcal{M})=\mbox{\rm codim}V(\mathcal{M'})=t$. Suppose that $F_{\mathcal{M}}^{p}:=\{F_{1}^{p},\ldots,F_{t}^{p}\}$, $F_{\mathcal{M'}}^{q}:=\{F_{1}^{'q},\ldots,F_{t}^{'q}\}$ are the global frames of $E_{\mathcal{M}},E_{\mathcal{M'}}$ corresponding to the generators $\{p_{1},\ldots,p_{t}\},\{q_{1},\ldots,q_{t}\}$, respectively. If the modules $\mathcal{M}$ and $\mathcal{M'}$ are "unitarily" equivalent, then there exists a constant invertible matrix $A$ such that
$$\mathcal{K}_{E_{\mathcal{M}}}(F_{\mathcal{M}}^{p})=(A^{*})^{-1}\cdot
\mathcal{K}_{E_{\mathcal{M'}}}(F_{\mathcal{M'}}^{q})\cdot A^{*},$$
where $\mathcal{K}_{E_{\mathcal{M}}}(F_{\mathcal{M}}^{p}),\mathcal{K}_{E_{\mathcal{M'}}}(F_{\mathcal{M'}}^{q})$ are the curvature matrices of $E_{\mathcal{M}},E_{\mathcal{M'}}$ with respect to the frames $F_{\mathcal{M}}^{p},F_{\mathcal{M'}}^{q}$, respectively.
\end{cor}
\begin{proof}
If we apply Theorem \ref{mainthm} on $\mathcal{M}$ with respect to the generator $\{q_{1},\ldots,q_{t}\}$, we will obtain a collection of anti-holomorphic maps $\{F_{1}^{q},\ldots,F_{t}^{q}\}$ from $V(\mathcal{M})$ to $\mathcal{M}$. This set canonically induces a global frame of $E_{\mathcal{M}}$. Let us denote the frame by $F_{\mathcal{M}}^{q}$. Then, by Proposition \ref{prop1} it follows that there exists a constant invertible matrix $A$ such that
$$\mathcal{K}_{E_{\mathcal{M}}}(F_{\mathcal{M}}^{p})=(A^{*})^{-1}
\cdot\mathcal{K}_{E_{\mathcal{M}}}(F_{\mathcal{M}}^{q})\cdot A^{*}.$$
Finally, from Theorem \ref{thm2} we obtain that
$$\mathcal{K}_{E_{\mathcal{M}}}(F_{\mathcal{M}}^{q})=\mathcal{K}_{E_{\mathcal{M'}}}(F_{\mathcal{M'}}^{q})$$
which proves the corollary.
\end{proof}

\medskip

\noindent{\textbf{Data availability}} Data sharing is not applicable to this article as no data sets were generated or analysed
during the current study.

\providecommand{\bysame}{\leavevmode\hbox to3em{\hrulefill}\thinspace}
\providecommand{\MR}{\relax\ifhmode\unskip\space\fi MR }
\providecommand{\MRhref}[2]{%
  \href{http://www.ams.org/mathscinet-getitem?mr=#1}{#2}
}
\providecommand{\href}[2]{#2}

\end{document}